\documentclass[11pt]{amsart}

\usepackage[T1]{fontenc}
\usepackage{lmodern}
\usepackage{amssymb,mathtools,bm}
\usepackage{booktabs}
\usepackage{graphicx}
\usepackage{float}
\usepackage{tikz-cd}
\usepackage{xcolor}
\usepackage{placeins}
\usepackage{microtype}
\usepackage{enumitem}
\usepackage[numbers,sort&compress]{natbib}
\usepackage[colorlinks=true,linkcolor=blue!45!black,
  citecolor=green!35!black,urlcolor=blue!55!black]{hyperref}
\usepackage[nameinlink,noabbrev]{cleveref}

\graphicspath{{figures/}}
\setlist{nosep}
\allowdisplaybreaks
\numberwithin{equation}{section}

\newtheorem{theorem}{Theorem}[section]
\newtheorem{lemma}[theorem]{Lemma}

\newtheorem{corollary}[theorem]{Corollary}
\theoremstyle{definition}

\newcommand{\R}{\mathbb R}
\newcommand{\N}{\mathbb N}
\newcommand{\dd}{\mathrm d}
\newcommand{\id}{\mathrm{id}}
\newcommand{\curl}{\operatorname{curl}}

\newcommand{\norm}[2][]{\left\lVert #2\right\rVert_{#1}}
\newcommand{\inner}[3][]{\left(#2,#3\right)_{#1}}

\title[Finite Element de Rham complexes on Sparse Grids]
{Higher-Order Finite Element de Rham Complexes on Sparse Grids}
\author{Yihanqi Hu}
\address{School of Mathematical Sciences, Zhejiang University, 866 Yuhangtang Road,
Hangzhou 310058, Zhejiang, People's Republic of China}
\email{huyihanqi@zju.edu.cn}
\author{Yuwen Li}
\address{School of Mathematical Sciences, Zhejiang University, 866 Yuhangtang Road,
Hangzhou 310058, Zhejiang, People's Republic of China}
\email{liyuwen@zju.edu.cn}
\date{}


\begin{document}

\begin{abstract}
We construct, for the first time, a family of higher-order finite element differential forms on tensor-product sparse grids.  The construction starts from
compatible one-dimensional spaces of continuous piecewise polynomials and
discontinuous piecewise polynomials of one degree lower, linked by
differentiation in each coordinate direction. Their hierarchical decompositions combine Alpert multiwavelets with their integrated counterparts. 
We establish commuting canonical interpolation operators and corresponding approximation error bounds under mixed Sobolev regularity.
For the sparse-grid de Rham complex of arbitrary polynomial degree on the unit cube in arbitrary dimension, we also prove exactness, polynomial-degree-robust
stable discrete potentials, and an $H(\dd)$-bounded commuting projection by developing a novel stable homotopy operator.
Numerical experiments for
curl-curl source problems on a cube and a Maxwell eigenproblem on a square-annulus
illustrate the effectiveness of the proposed higher-order sparse-grid method.
\end{abstract}

\maketitle

\section{Introduction}
\label{sec:introduction}

Tensor-product finite elements are a natural choice on Cartesian meshes:
their local spaces, interpolation operators, and discrete differential
operators inherit a transparent tensor-product structure from one dimension. They inevitably suffer from the curse of dimensionality.  If the mesh size in each direction is $h_n=O(2^{-n})$, a full-grid tensor-product finite element space would have $O(h_n^{-d})$ degrees of freedom (DoFs) in
dimension $d$.  Sparse-grid methods drastically reduce the number of DoFs to
$O(h_n^{-1}|\log h_n|^{d-1})$ through hyperbolic-cross approximation.  For functions with sufficient mixed Sobolev regularity, this
reduction preserves the approximation power of full-grid finite elements up to logarithmic factors; see, e.g., \cite{BungartzGriebel2004,DungTemlyakovUllrich2018} for sparse-grid continuous finite elements and \cite{ShenYu2010} for a sparse-grid spectral method.

Sparse-grid discretization is
well established for scalar-valued function approximation, but its extension to vector fields in $H(\curl)$ and $H(\operatorname{div})$ remains underdeveloped. Because of the delicate numerical stability requirements, a naive componentwise sparse-grid discretization of these vector fields can lead to well-known spurious numerical solutions for curl-curl source and eigenvalue problems on domains with holes or reentrant corners. In \cite{GradinaruHiptmair2003,GradinaruHiptmair2003Multigrid}, Gradinaru and Hiptmair developed sparse-grid finite elements in $H(\curl)$, $H(\operatorname{div})$, and, more generally, spaces $H\Lambda^l$ of differential $l$-forms. Their construction is restricted to the lowest polynomial order and hinges on the decisive 1D hat--Haar relation:
\[
\frac{\dd\phi_i}{\dd x}=\psi_i,
\]
where $\phi_i$ is a hierarchical hat function associated with the $i$th grid node $a_i$, and $\psi_i$ is its derivative, a piecewise-constant Haar wavelet. At higher polynomial degree, replacing hat functions and Haar
wavelets by arbitrary hierarchical polynomial bases does not retain this
property. Besides sparse grid, several serendipity conforming finite elements on cubical meshes can achieve milder reduction of DoFs \cite{ArnoldAwanou2011,ArnoldAwanou2014,ArnoldBoffiBonizzoni2015}.

In a parallel development, Alpert multiwavelets \cite{Alpert1993} have
been used to construct high-order sparse-grid discontinuous Galerkin (DG) spaces for
elliptic and transport problems
\citep{WangTangGuoCheng2016,GuoCheng2016,GuoCheng2017}. Sparse-grid DG methods have also been developed for solving Maxwell and high-dimensional Maxwell--Vlasov equations \citep{DAzevedoGreenMu2020,TaoGuoCheng2019}.  Their spaces, however, are based on scalar or componentwise sparse-grid DG approximation for vector fields
and do not preserve the structure of the de Rham complex at the discrete level. In such cases, the numerical stability of discontinuous sparse-grid methods is achieved by adding facewise interior-penalty stabilization terms.

The purpose of this paper is to construct and analyze higher-order conforming sparse-grid discretizations that preserve structural properties of the de Rham complex, including exactness and the existence of stable discrete potentials. Given a dyadic 1D mesh $\mathcal{T}_m$ in each direction, we
start from the 1D finite element de Rham complex
\begin{equation}
  S_m^{k+1}
  \xrightarrow{\dd/\dd x}
  Q_m^{k},
  \label{eq:intro-1D-pair}
\end{equation}
where $S_m^{k+1}$ is the continuous piecewise-$\mathbb{P}_{k+1}$ space and
$Q_m^{k}$ is the discontinuous piecewise-$\mathbb{P}_k$ space.  The detail space
of $Q_m^{k}$ is represented by Alpert multiwavelets.  Integrating
them produces continuous detail functions in $S_m^{k+1}$ whose
derivatives are exactly the original Alpert functions. This is the
higher-order counterpart of the hat--Haar relation. Tensorizing the 1D details then yields multidimensional $H\Lambda^l$-conforming sparse-grid finite element differential forms in arbitrary dimension and form degree. We remark that the higher-order continuous hierarchy constructed here differs from the hierarchical Lagrange hierarchy given in \cite{BungartzGriebel2004}. It is difficult to pair the continuous 1D hierarchy in \cite{BungartzGriebel2004} with the discontinuous 1D hierarchy in \cite{Alpert1993} to derive a higher-order hat--Haar relation.

It is instructive to emphasize the
necessity of an abstract differential-form framework, namely finite element exterior calculus (FEEC) \citep{Hiptmair1999,ArnoldFalkWinther2006,Arnold2018}. First, sparse-grid methods are intended to solve PDEs on high-dimensional domains ($3\leq d\leq10$). Therefore, differential forms, their exterior derivatives, and their traces provide the natural language because the classical vector-calculus proxies for curl and tangential traces do not extend uniformly to arbitrary dimension. Bounded commuting
projections in FEEC are equally important: they transfer differential identities
from the continuous problem to the discrete one and underlie the numerical stability
of mixed methods and Maxwell discretizations. In fact, we shall construct bounded commuting
projections on sparse grids by leveraging a novel stable homotopy operator and a tensor trick from algebraic and differential topology.

The main contributions of this paper and steps of analysis are as follows.
\begin{enumerate}[label=(\roman*)]
  \item We give compatible canonical interpolation operators and
  hierarchical decompositions for the 1D spaces in
  \eqref{eq:intro-1D-pair}.  The integrated-Alpert basis
  preserves continuity and satisfies an exact derivative identity with the
  Alpert basis, so the canonical and hierarchical descriptions define the
  same commuting interpolant. The continuous 1D hierarchy is new and different from \cite{BungartzGriebel2004}.

  \item We tensorize the 1D construction for arbitrary dimension $d$ and form degree $l$. The polynomial degree is at most $k$ in every form direction and at most $k+1$ in every transverse direction. Using one
  common downward-closed sparse index set in every form degree yields a
  discrete subcomplex and a commuting sparse-grid interpolant.  Moreover, we
  prove approximation error bounds for the canonical commuting interpolants under mixed Sobolev regularity.

  \item On a unit cube, we construct a novel explicit tensor-product
  homotopy that preserves every sparse-grid finite element space.  It proves discrete exactness
  and provides discrete potentials that are uniformly stable with respect to both grid level and polynomial degree. 
  These potentials lead to an $H(\dd)$-bounded commuting projection. Our technique differs from that of \cite{GradinaruHiptmair2003} even at the lowest polynomial order.
\end{enumerate}

The paper is organized as follows.  Section~\ref{sec:spaces} introduces the
tensor-product spaces and constructs the compatible 1D
hierarchies.  Sections~\ref{sec:sparse-interpolation} and~\ref{sec:approximation} develop the
sparse interpolant and its mixed-regularity error analysis.
Section~\ref{sec:stability} establishes exactness, stable potentials, and the commuting
projection.  Section~\ref{sec:numerics} presents the
numerical experiments, followed by the concluding remarks in
Section~\ref{sec:conclusion}.

\section{Finite element differential forms and hierarchies}
\label{sec:spaces}

Let $\Omega=(0,1)^d$, $\bm x=(x_1,\ldots,x_d)$, $[d]:=\{1,\ldots,d\}$, and $\N_0$ be the set of nonnegative integers.  For an
interval $K$ and $r\in\N_0$, let
$\mathbb{P}_r(K)$ denote the polynomials of degree at most $r$ on $K$, with the
convention $\mathbb{P}_{-1}(K)=\{0\}$.  We use $\mathbb{P}_r(x_i)$ to denote the space of
univariate polynomials of degree at most $r$ in the coordinate $x_i$.  For any unordered index set $A\subseteq[d]$,
let $A^c=[d]\setminus A$, and let
\[
  \bm x_A=(x_i)_{i\in A},
  \qquad
  \dd\bm x_A=\prod_{i\in A}\dd x_i
\]
denote the corresponding coordinate subvector and product measure.
For an ordered multi-index set $I=\{i_1<\cdots<i_l\}\subseteq[d]$, let
\[
  \dd x_I=\dd x_{i_1}\wedge\cdots\wedge\dd x_{i_l}
\]
be the
coordinate differential form. Then
any $l$-form can be written as
\[
  \omega=\sum_{I\subseteq[d],\,|I|=l}u_I(\bm x)\,\dd x_I.
\]
The space of such forms with coefficients $u_I\in L^2(\Omega)$ is denoted by
$L^2\Lambda^l(\Omega)$. The $L^2$ norm for $l$-forms is
\[
  \norm[L^2\Lambda^l(\Omega)]{\omega}^2
  =\sum_{|I|=l}\norm[L^2(\Omega)]{u_I}^2.
\]
For a fixed component $u_I\dd x_I$, we refer to the coordinates indexed by
$i\in I$ as the form directions, since $\dd x_i$ occurs in $\dd x_I$, and
to those indexed by $j\in I^c$ as the transverse directions.
The exterior derivative is
\begin{equation*}
  \dd\omega
  =\sum_{|I|=l}\sum_{j\notin I}
  \partial_{x_j}u_I\,\dd x_j\wedge\dd x_I,
  \qquad \dd^2=0.
\end{equation*}
The domain Sobolev space of $\dd$ (in the weak sense) is
\[
  H\Lambda^l(\Omega)
  =\{\omega\in L^2\Lambda^l(\Omega):
      \dd\omega\in L^2\Lambda^{l+1}(\Omega)\},
\]
which is equipped with the graph norm
\[\norm[H(\dd,\Omega)]{\omega}^2
  =\norm[L^2\Lambda^l(\Omega)]{\omega}^2+\norm[L^2\Lambda^{l+1}(\Omega)]{\dd\omega}^2.\]
In vector proxies, $H\Lambda^0=H^1$,
$H\Lambda^1=H(\curl)$ in two and three dimensions, and
$H\Lambda^2=H(\operatorname{div})$ in three dimensions
\citep{ArnoldFalkWinther2006,Monk2003}.

\subsection{Cubical finite element forms}
On the reference cube $\widehat K=[0,1]^d$, the shape-function space of an $H\Lambda^l$-conforming finite element is
\begin{equation}
  \mathcal{Q}_{k+1}^{-}\Lambda^l(\widehat K)
  =
  \bigoplus_{\substack{I\subseteq[d]\\ |I|=l}}
  \left(\bigotimes_{i\in I}\mathbb{P}_k(x_i)\right)
  \otimes
  \left(\bigotimes_{j\notin I}\mathbb{P}_{k+1}(x_j)\right)
  \dd x_I.
  \label{eq:trimmed-tensor-space}
\end{equation}
Thus a coordinate in a form direction has degree $k$, whereas a transverse
coordinate has degree $k+1$.  This is the tensor-product space of discrete
polynomial differential forms described in finite element
exterior calculus  \citep{ArnoldBoffiBonizzoni2015,GilletteKloefkornSanders2019}.  The lowest-order
case $k=0$ is the tensor-product Whitney family underlying the sparse-grid
constructions of
\citet{GradinaruHiptmair2003,GradinaruHiptmair2003Multigrid}.  In two dimensions,
\begin{equation*}
  \mathcal{Q}_{k+1}^{-}\Lambda^1(\widehat K)
  =
  \bigl(\mathbb{P}_k(x_1)\otimes\mathbb{P}_{k+1}(x_2)\bigr)\dd x_1
  \oplus
  \bigl(\mathbb{P}_{k+1}(x_1)\otimes\mathbb{P}_k(x_2)\bigr)\dd x_2,
\end{equation*}
whose vector proxy is the rectangular edge finite element \citep{Nedelec1980,Monk2003}. 

For completeness, we briefly explain the degrees of freedom assigned to the shape function space in \eqref{eq:trimmed-tensor-space}.  Fix a multi-index
$I$, and let
$R$ range over all index sets satisfying $I\subseteq R\subseteq[d]$.  The
coordinates in $R$ vary on an $|R|$-dimensional coordinate face.
For $\bm\epsilon=(\epsilon_j)_{j\in R^c}\in\{0,1\}^{R^c}$, define
\[
  F_{R,\bm\epsilon}
  :=\{\bm x\in[0,1]^d:x_j=\epsilon_j\text{ for every }j\in R^c\}.
\]
On that face choose
\begin{equation}
  p(\bm x_R)\in
  \left(\bigotimes_{i\in I}\mathbb{P}_k(x_i)\right)
  \otimes
  \left(\bigotimes_{j\in R\setminus I}\mathbb{P}_{k-1}(x_j)\right).
  \label{eq:test-polynomial-space}
\end{equation}
After fixing bases in these test spaces, the degrees of freedom are the
moments
\begin{equation}
  D_{F_{R,\bm\epsilon},I,p}(\omega)
  =
  \int_{F_{R,\bm\epsilon}}
  \operatorname{tr}_{F_{R,\bm\epsilon}}\omega
  \wedge p(\bm x_R)\,\dd x_{R\setminus I}.
  \label{eq:geometric-dofs}
\end{equation}
Here $\operatorname{tr}_{F_{R,\bm\epsilon}}$ denotes the pullback, or
tangential trace, of a differential form to the face $F_{R,\bm\epsilon}$.
Up to the orientation sign, the contribution of the $I$th component is the
more familiar scalar moment
\[
  \int_{F_{R,\bm\epsilon}}
  u_I(\bm x_R,\bm\epsilon)p(\bm x_R)\,\dd\bm x_R.
\]
When $R=I$, these are tangential moments on $l$-faces.  Larger $R$ gives
higher-dimensional face and cell-interior moments.  These functionals are
the standard tensor-product degrees of freedom, not a new moment system.
Indeed, set $s=k+1$.  On an interval, the $0$-form space $\mathbb{P}_s$ has
endpoint values and interior moments against $\mathbb{P}_{s-2}$, while the
$1$-form space $\mathbb{P}_{s-1}\dd x$ has moments against $\mathbb{P}_{s-1}$.  Taking
their tensor products gives an unisolvent system whose invariant form on an
$r$-face $F$ is
$\omega\mapsto\int_F\operatorname{tr}_F\omega\wedge q$ with
$q\in\mathcal{Q}_{s-1}^{-}\Lambda^{r-l}(F)$.  For the component $\dd x_I$, taking
$q=p\,\dd x_{R\setminus I}$ gives exactly
\eqref{eq:test-polynomial-space}--\eqref{eq:geometric-dofs}; see
\citet{ArnoldBoffiBonizzoni2015}.

On a conforming cubical mesh, degrees of freedom attached to a common
geometric face are identified with a consistent orientation.  This assembly
enforces the trace continuity required by $H\Lambda^l$.  In particular,
the 1D factors in form directions are discontinuous
piecewise-$\mathbb{P}_k$ functions, whereas transverse factors assemble as
continuous piecewise-$\mathbb{P}_{k+1}$ functions.  This observation is the
starting point for the compatible multilevel construction below.

\subsection{Compatible one-dimensional hierarchy}
\label{sec:hierarchy}

Let $\mathcal{T}_m$ be the uniform dyadic partition of $(0,1)$ into $2^m$ cells.
Define the nested spaces
\begin{align*}
  S_m^{k+1}
  &=
  \{v\in C^0([0,1]):v|_K\in\mathbb{P}_{k+1}(K),
    \ K\in\mathcal{T}_m\},\\
  Q_m^{k}
  &=
  \{q\in L^2(0,1):q|_K\in\mathbb{P}_k(K),
    \ K\in\mathcal{T}_m\}.
\end{align*}
Throughout the paper, $S_m^{k+1}$ and $Q_m^{k}$, and likewise their
detail spaces introduced below, denote scalar coefficient spaces.  In the 1D de Rham complex, $S_m^{k+1}$ is a $0$-form space and $Q_m^{k}\,\dd x$ is a $1$-form space.
They satisfy
\begin{equation*}
  \frac{\dd}{\dd x}S_m^{k+1}=Q_m^{k}.
\end{equation*}
For each fixed polynomial degree $k$, we shall suppress the $k$-dependence of
all interpolation and detail operators defined later.  

\subsubsection{Canonical interpolation and commutation}

On $K=[a,b]$, the continuous finite element triple
$\bigl(K,\mathbb{P}_{k+1}(K),\Sigma_K^0\bigr)$ uses the DoF set $\Sigma_K^0$ of endpoint values
$v\mapsto v(a)$, $v\mapsto v(b)$ and the $k$ interior moments
\[
  v\mapsto\int_K v(x)w(x)\,\dd x,\qquad w\in\mathbb{P}_{k-1}(K).
\]
The discontinuous finite element triple
$\bigl(K,\mathbb{P}_k(K),\Sigma_K^1\bigr)$ uses all $k+1$ interior moments as its DoF set $\Sigma_K^1$:
\[
  q\mapsto\int_K q(x)p(x)\,\dd x,\qquad p\in\mathbb{P}_k(K).
\]
For \(v\in H^1(0,1)\) and
\(q\in L^2(0,1)\), let
$\Pi_m^0v\in S_m^{k+1}$ and
$\Pi_m^1q\in Q_m^{k}$ denote the global canonical interpolants assembled from these local DoFs.  The latter is the $L^2$
projection. For each $K\in\mathcal T_m$ and $q\in\mathbb{P}_k(K)$, integration by parts gives
\[
  \int_K\bigl((\Pi_m^0v)'-v'\bigr)q\,\dd x
  =
  \bigl[(\Pi_m^0v-v)q\bigr]_{a}^{b}
  -\int_K(\Pi_m^0v-v)q'\,\dd x=0.
\]
Thus $(\Pi_m^0v)'$ and $v'$ share the
same moments against every $q\in\mathbb{P}_k(K)$ and
\begin{equation}
  (\Pi_m^0v)'=\Pi_m^1v'.
  \label{eq:1D-commutation}
\end{equation}
This commuting identity is the source of every
commuting relation later.

\subsubsection{Alpert and integrated-Alpert hierarchies}

Set $Z_0^{k}=Q_0^{k}$.  For $m\ge1$, define the $L^2$-orthogonal detail
space
\begin{equation*}
  Z_m^{k}
  =Q_m^{k}\cap Q_{m-1}^{k,\perp}.
\end{equation*}
For every $L\ge0$, the following direct sum is $L^2$-orthogonal:
\begin{equation*}
  Q_L^{k}=\bigoplus_{m=0}^{L} Z_m^{k}.
\end{equation*}
The construction of each hierarchical component is local. Let $K\in\mathcal{T}_{m-1}$ be split into two child intervals
$K_L$ and $K_R$, and define
\begin{equation*}
  \begin{aligned}
  Z_K^{k}=\bigl\{z\in L^2(K):\;&
  z|_{K_L}\in\mathbb{P}_k(K_L),\quad z|_{K_R}\in\mathbb{P}_k(K_R),\\
  &\int_K zq\,\dd x=0\quad\forall q\in\mathbb{P}_k(K)\bigr\}.
  \end{aligned}
\end{equation*}
Extending the local functions by zero outside their parent cells gives
\[
  Z_m^k=\bigoplus_{K\in\mathcal{T}_{m-1}}Z_K^k,\qquad m\ge1.
\]
Each $Z_K^k$ has dimension $k+1$, and hence
$\dim Z_m^k=2^{m-1}(k+1)$.  Choose an $L^2(K)$-orthonormal basis
$\{\psi_{m,K,r}^{k}\}_{r=0}^k$ of each local space.  These form an Alpert-type
polynomial multiwavelet basis \citep{Alpert1993}; related high-order sparse-grid
DG constructions use the same orthogonal polynomial hierarchy
\citep{WangTangGuoCheng2016,GuoCheng2016,GuoCheng2017}.

For compact notation, let $\Lambda_0^k=\{0,\ldots,k\}$ and, for $m\ge1$,
\[
  \Lambda_m^k
  =\{(K,r):K\in\mathcal{T}_{m-1},\ 0\le r\le k\}.
\]
We write $\psi_{m,\alpha}^{k}=\psi_{m,K,r}^{k}$ with
$\alpha=(K,r)\in\Lambda_m^k$.

For $m\ge1$, $\alpha=(K,r)\in\Lambda_m^k$, and $K=[a_K,b_K]$, define
\begin{equation*}
  \theta_{m,\alpha}^{k+1}(x)
  =
  \begin{cases}
    \displaystyle\int_{a_K}^{x}\psi_{m,\alpha}^{k}(s)\,\dd s,
      &x\in K,\\[1ex]
    0,&x\notin K.
  \end{cases}
\end{equation*}
Since $1\in\mathbb{P}_k(K)$, every element of $Z_K^k$ has zero mean on its parent.
Thus $\theta_{m,\alpha}^{k+1}$ vanishes at both endpoints of $K$. As a result, $\theta_{m,\alpha}^{k+1}$ is continuous and belongs to $S_m^{k+1}$.  Set
\[
  Y_m^{k+1}
  =\operatorname{span}\big\{\theta_{m,\alpha}^{k+1}:
             \alpha\in\Lambda_m^k\big\},\qquad m\ge1.
\]
At level zero, set $\widehat I=[0,1]$ and choose an orthonormal basis
$\{\psi_{0,r}^{k}\}_{r=0}^k$ of $\mathbb{P}_k(\widehat I)$; in the numerical
implementation, this is the normalized shifted Legendre basis.  Define
\[
  \theta_{0,\star}^{k+1}=1,\qquad
  \theta_{0,r}^{k+1}(x)
  =\int_0^x\psi_{0,r}^{k}(s)\,\dd s.
\]
These functions span $Y_0^{k+1}=S_0^{k+1}$.  Every nonconstant mode
satisfies
\begin{equation*}
  \frac{\dd}{\dd x}\theta_{m,\alpha}^{k+1}
  =\psi_{m,\alpha}^{k}.
\end{equation*}

\begin{lemma}[Continuous hierarchy]  \label{lem:continuous-hierarchy}
For every $L\ge0$,
\begin{equation*}
  S_L^{k+1}=\bigoplus_{m=0}^{L}Y_m^{k+1}.
\end{equation*}
\end{lemma}

\begin{proof}
If $g\in S_{m-1}^{k+1}\cap Y_m^{k+1}$, then
$g'\in Q_{m-1}^{k}\cap Z_m^{k}=\{0\}$, and thus $g$ is constant.
Every function in $Y_m^{k+1}$ with $m\ge1$ vanishes at all parent nodes,
so $g=0$. Together with the dimension count
\[
  \dim Y_m^{k+1}
  =2^{m-1}(k+1)
  =\dim S_m^{k+1}-\dim S_{m-1}^{k+1}.
\]
the identity $S_{m-1}^{k+1}\cap Y_m^{k+1}=\{0\}$ proves the result by induction.
\end{proof}

By $L^2$-orthogonality, the canonical interpolation $\Pi_L^1$ can be written hierarchically as
\begin{equation*}
  \Pi_L^1q
  =\sum_{m=0}^{L}\sum_{\alpha\in\Lambda_m^k}
    d_{m,\alpha}\psi_{m,\alpha}^{k},
  \qquad
  d_{m,\alpha}
  =\inner[L^2(0,1)]{q}{\psi_{m,\alpha}^{k}}.
\end{equation*}
For the one-dimensional continuous interpolant onto $S_L^{k+1}$, we write
\begin{equation}
  \Pi_L^0v
  =v(0)\theta_{0,\star}^{k+1}
   +\sum_{m=0}^{L}\sum_{\alpha\in\Lambda_m^k}
    b_{m,\alpha}\theta_{m,\alpha}^{k+1},\qquad   b_{m,\alpha}
  =\inner[L^2(0,1)]{v'}{\psi_{m,\alpha}^{k}}.
  \label{eq:continuous-hierarchical-expansion}
\end{equation}
The formula for $b_{m,\alpha}$ is derived by differentiating
\eqref{eq:continuous-hierarchical-expansion}, using the commuting identity
\eqref{eq:1D-commutation}, and testing with the orthonormal
Alpert basis $\psi_{m,\alpha}^k$.
At lower
regularity, the finite element degrees of freedom still determine
$\Pi_L^0v$ and hence its unique coordinates in Lemma
\ref{lem:continuous-hierarchy}, provided the endpoint traces and moments
are bounded on the chosen domain.

\begin{figure}[!tbp]
  \centering
  \includegraphics[width=0.98\linewidth]{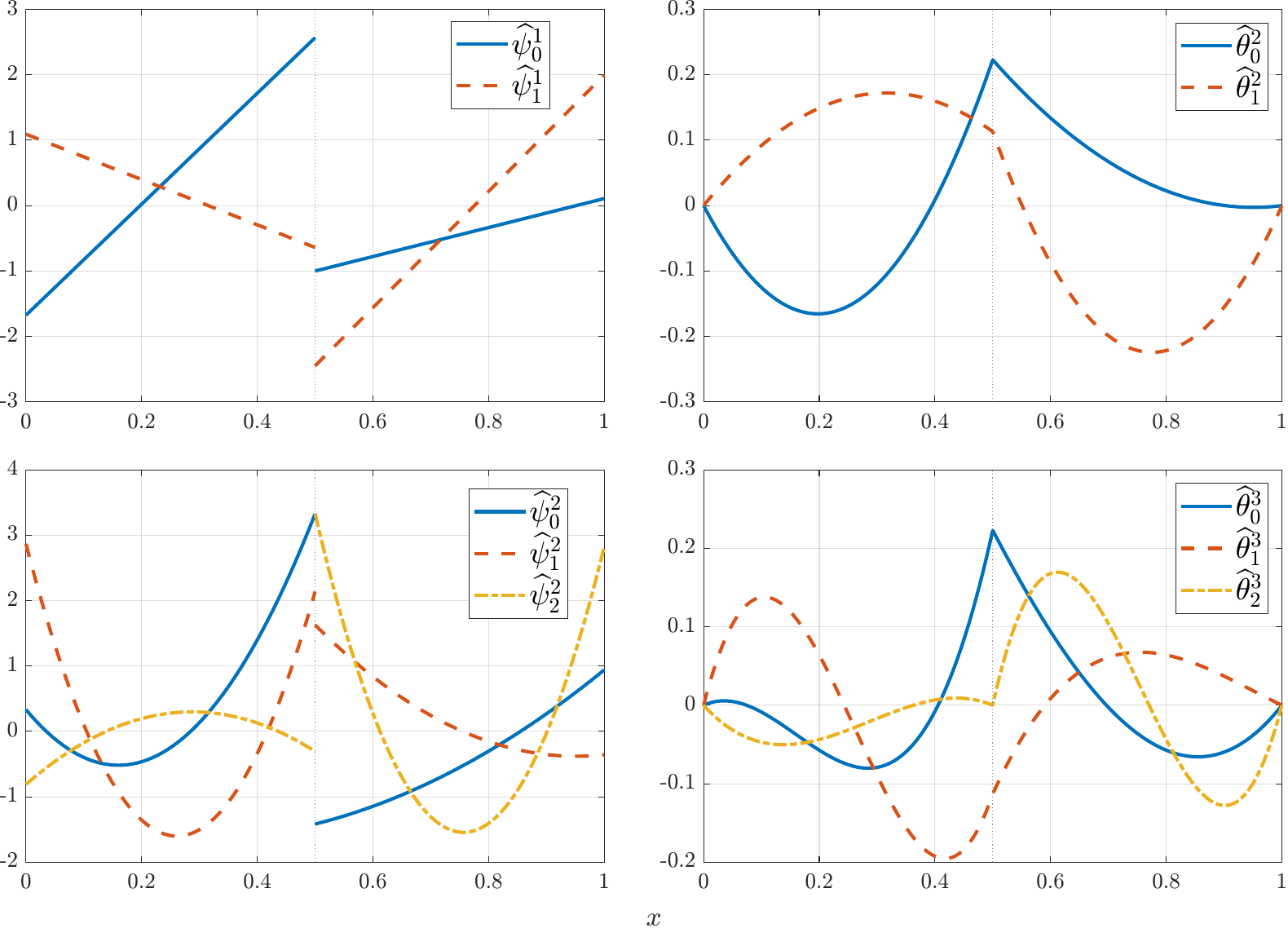}
  \caption{1D bases used in the computations on the reference
  interval $\widehat I=[0,1]$.  The top and bottom rows
  correspond to $k=1$ and $k=2$, respectively.  The left panels show the
  discontinuous functions $\widehat\psi_r^k$, and the right
  panels show their integrated continuous partners
  $\widehat\theta_r^{k+1}$.}
  \label{fig:one-dimensional-bases}
\end{figure}

\subsubsection{Computation of hierarchical basis}

The orthonormal basis $\{\psi_{m,K,r}^{k}\}_{r=0}^k$ of $Z_K^k$ is not unique.  For the numerical
implementation, a particular basis is constructed once on the reference
interval $\widehat I$ as follows.  Let
$\{\phi_q\}_{q=0}^k$, with $\phi_q=\psi_{0,q}^k$, be the normalized shifted Legendre basis of
$\mathbb P_k(\widehat I)$, and let
$\{\eta_s\}_{s=0}^{2k+1}$ be obtained from the normalized shifted Legendre
bases on $\widehat I_L=[0,1/2]$ and $\widehat I_R=[1/2,1]$ by extending
each basis function by zero to $\widehat I$.  Define the coupling matrix
\[
  M_{qs}=\int_{\widehat I}\phi_q(x)\eta_s(x)\,\dd x,
  \qquad 0\le q\le k,\quad 0\le s\le 2k+1.
\]
Let $B$ be a matrix whose
columns form an orthonormal basis of the null space of $M$.  The resulting reference
detail functions are
\[
  \widehat\psi_r^k
  =\sum_{s=0}^{2k+1}B_{sr}\eta_s,
  \qquad r=0,\ldots,k.
\]
They form an $L^2(\widehat I)$-orthonormal basis of the reference detail
space.  Define their integrated partners by
\[
  \widehat\theta_r^{k+1}(s)
  =\int_0^s \widehat\psi_r^k(t)\,\dd t.
\]
For $m\ge1$ and $\alpha=(K,r)\in\Lambda_m^k$, with
$K=[a_K,b_K]$, let $F_K(t)=a_K+|K|t$ map $\widehat I$ onto $K$.
The corresponding basis functions are obtained by the affine transformations
\begin{align*}
   \psi_{m,\alpha}^k(x)
  &=|K|^{-1/2}\widehat\psi_r^k
     \bigl(F_K^{-1}(x)\bigr),
  \qquad x\in K,\\
  \theta_{m,\alpha}^{k+1}(x)
  &=|K|^{1/2}\widehat\theta_r^{k+1}
     \bigl(F_K^{-1}(x)\bigr),
  \qquad x\in K,
\end{align*}
with both functions extended by zero outside $K$.  Thus the null space is
computed only once for each $k$.  The resulting reference bases are shown in
Figure~\ref{fig:one-dimensional-bases}.

\section{Tensor-product interpolation and sparse truncation}
\label{sec:sparse-interpolation}
In this section, we present the construction of finite element spaces of differential forms on sparse grids.

\subsection{Tensor-product detail spaces}

Set $\bm 0:=(0,\ldots,0)\in\N_0^d$.  For a multi-index
$\bm \nu=(\nu_1,\ldots,\nu_d)\in\N_0^d$, we write
$|\bm \nu|_1:=\sum_{s=1}^d \nu_s$, $|\bm \nu|_\infty:=\max_{1\le s\le d}\nu_s.$
For $\bm \mu,\bm \nu\in\N_0^d$, we say $\bm \mu\le\bm \nu$ if $
  \mu_s\le \nu_s$ for every $s\in[d]$.
A multi-index
$\bm\nu=(\nu_1,\ldots,\nu_d)\in\N_0^d$ determines an anisotropic level and the tensor grid
\[
  \mathcal{T}_{\bm\nu}
  :=\mathcal{T}_{\nu_1}\otimes\cdots\otimes\mathcal{T}_{\nu_d}.
\]
Here the interval (0,1) in the $x_i$-direction is uniformly bisected $\nu_i$
times and hence consists of $2^{\nu_i}$ cells.
For $s\in[d]$, $r\in\{0,1\}$, and $m\in\N_0$, let $\Pi_{s,m}^r$ denote
the 1D operator $\Pi_m^r$ acting only in the variable $x_s$.
For each component $u_I\dd x_I$, we define
\begin{equation}\label{eq:tensor-canonical-interpolation}
      \Pi_{\bm\nu}^{I}u_I
  =
  \left(\prod_{i\in I}\Pi_{i,\nu_i}^1\right)
  \left(\prod_{j\notin I}\Pi_{j,\nu_j}^0\right)u_I.
\end{equation}
Clearly, interpolation operators in each direction commute, and the order of the products is
irrelevant; an empty product is understood as the identity.
For a general differential form $\omega=\sum_{|I|=l}u_I\dd x_I$, define
\begin{equation}
  \Pi_{\bm\nu}^{l}\omega
  =
  \sum_{|I|=l}\Pi_{\bm\nu}^{I}u_I\,\dd x_I.
\end{equation}
This is the canonical interpolant associated with the geometric DoFs in
\eqref{eq:geometric-dofs}. For an ordered set $I$, the notation
$I\cup\{j\}$ always means the increasing ordering of the enlarged index
set.  We then write
\[
  \dd(u_I\dd x_I)
  =
  \sum_{j\notin I}\epsilon(j,I)\,
  \partial_{x_j}u_I\,\dd x_{I\cup\{j\}},
  \qquad
  \epsilon(j,I)=(-1)^{\#\{i\in I:i<j\}}.
\]

For completeness,
we verify the commuting property $\dd\Pi_{\bm\nu}^{l}
  =\Pi_{\bm\nu}^{l+1}\dd$. Let $j\notin I$ be an index for a transverse direction.
In \eqref{eq:tensor-canonical-interpolation}, the $x_j$-factor is
$\Pi_{j,\nu_j}^0$, whereas the factors in all other coordinates act on
variables different from $x_j$ and hence commute with
$\partial_{x_j}$.  Therefore, using the 1D relation
\eqref{eq:1D-commutation},
\[
  \begin{aligned}
  \partial_{x_j}\Pi_{\bm\nu}^{I}u_I
  &=
  \left(\prod_{i\in I}\Pi_{i,\nu_i}^1\right)
  \left(\partial_{x_j}\Pi_{j,\nu_j}^0\right)
  \left(\prod_{s\notin I\cup\{j\}}\Pi_{s,\nu_s}^0\right)u_I\\
  &=
  \left(\prod_{i\in I\cup\{j\}}\Pi_{i,\nu_i}^1\right)
  \left(\prod_{s\notin I\cup\{j\}}\Pi_{s,\nu_s}^0\right)
  \partial_{x_j}u_I\\
  &=\Pi_{\bm\nu}^{I\cup\{j\}}\partial_{x_j}u_I.
  \end{aligned}
\]
Applying the exterior derivative to the interpolated differential-form
component now gives
\[
  \begin{aligned}
  \dd\left[
    \left(\Pi_{\bm\nu}^{I}u_I\right)\dd x_I
  \right]
  =
  \sum_{j\notin I}\epsilon(j,I)
  \Pi_{\bm\nu}^{I\cup\{j\}}
  \partial_{x_j}u_I\,\dd x_{I\cup\{j\}}
  =
  \Pi_{\bm\nu}^{l+1}\dd(u_I\dd x_I).
  \end{aligned}
\]
The last equality is the componentwise definition of
$\Pi_{\bm\nu}^{l+1}$ applied to $\dd(u_I\dd x_I)$.

Finally, summing the preceding identity over all $|I|=l$ yields
\[
  \begin{aligned}
  \dd\Pi_{\bm\nu}^{l}\omega
  &=
  \sum_{|I|=l}
  \dd\left[
    \left(\Pi_{\bm\nu}^{I}u_I\right)\dd x_I
  \right]\\
  &=
  \sum_{|I|=l}
  \Pi_{\bm\nu}^{l+1}\dd(u_I\dd x_I)
  =
  \Pi_{\bm\nu}^{l+1}\dd\omega.
  \end{aligned}
\]
Since this identity holds for every such sufficiently regular $l$-form
$\omega$, we obtain, for $0\le l\le d-1$,
\begin{equation*}
  \dd\Pi_{\bm\nu}^{l}
  =\Pi_{\bm\nu}^{l+1}\dd.
\end{equation*}

Set $\Pi_{-1}^0=\Pi_{-1}^1=0$ and define the one-dimensional differences
\[
  \Delta_m^r=\Pi_m^r-\Pi_{m-1}^r,
  \qquad r\in\{0,1\},\quad m\in\N_0.
\]
Their ranges are $Y_m^{k+1}$ and $Z_m^{k}$, respectively.  For
$s\in[d]$, let $\Delta_{s,m}^r$ denote $\Delta_m^r$ acting only in the variable $x_s$.  For
$I\subseteq[d]$ and $\bm m=(m_1,\ldots,m_d)\in\N_0^d$, define the scalar
coefficient detail block
\begin{subequations}
  \begin{align}
  W_{\bm m}^{I,k}
  &=
  \left(\bigotimes_{i\in I}Z_{m_i}^{k}(x_i)\right)
  \otimes
  \left(\bigotimes_{j\notin I}Y_{m_j}^{k+1}(x_j)\right),\label{eq:WmI}
  \\
  H_{\bm m}^{I}
  &=
  \left(\prod_{i\in I}\Delta_{i,m_i}^1\right)
  \left(\prod_{j\notin I}\Delta_{j,m_j}^0\right).
\end{align}
\end{subequations}
The tensor telescoping identity gives
\begin{equation*}
  \Pi_{\bm\nu}^{I}
  =
  \sum_{\bm 0\le\bm m\le\bm\nu}H_{\bm m}^{I},
  \qquad
  \bigotimes_{i\in I}Q_{\nu_i}^{k}(x_i)
  \otimes\bigotimes_{j\notin I}S_{\nu_j}^{k+1}(x_j)
  =
  \bigoplus_{\bm 0\le\bm m\le\bm\nu}W_{\bm m}^{I,k}.
\end{equation*}

To write a basis, extend the continuous index set by
\[
  \widetilde\Lambda_0^k=\{\star\}\cup\Lambda_0^k,\qquad
  \widetilde\Lambda_m^k=\Lambda_m^k\quad(m\ge1).
\]
For $\bm\alpha\in\Lambda_{\bm m}^{I,k}
  :=
  \left(\prod_{i\in I}\Lambda_{m_i}^k\right)
  \times
  \left(\prod_{j\notin I}\widetilde\Lambda_{m_j}^k\right)$, we
define
\begin{equation*}
  \Psi_{\bm m,\bm\alpha}^{I,k}(\bm x)
  =
  \prod_{i\in I}\psi_{m_i,\alpha_i}^{k}(x_i)
  \prod_{j\notin I}\theta_{m_j,\alpha_j}^{k+1}(x_j).
\end{equation*}
These scalar functions form a basis of $W_{\bm m}^{I,k}$, and
$\{\Psi_{\bm m,\bm\alpha}^{I,k}\dd x_I\}_{\bm\alpha}$ is the associated
basis of the form block $W_{\bm m}^{I,k}\dd x_I$.

\subsection{Sparse spaces and commuting interpolation}

For $n\in\N_0$, let
\begin{equation*}
  \mathcal{M}_n
  =
  \{\bm m\in\N_0^d:|\bm m|_1\le n+d-1\}.
\end{equation*}
For a fixed form index $I$, define the scalar coefficient truncation
\[
  \mathcal I_n^{I}u_I
  :=
  \sum_{\bm m\in\mathcal{M}_n}H_{\bm m}^{I}u_I.
\]
The sparse-grid space of $l$-forms is
\begin{equation}
  V_{n,k}^l
  =
  \bigoplus_{I\subseteq[d],\, |I|=l}
  \bigoplus_{\bm m\in\mathcal{M}_n}
  W_{\bm m}^{I,k}\,\dd x_I.
  \label{eq:sparse-form-space}
\end{equation}
The corresponding sparse-grid interpolation onto $V_{n,k}^l$ is
\begin{equation*}
  \mathcal I_n^l\omega
  =
  \sum_{I\subseteq[d],\, |I|=l}
  \left(\mathcal I_n^{I}u_I\right)\dd x_I.
\end{equation*}
The same index set $\mathcal{M}_n$ is used in every form degree.  This point is
essential: differentiation changes a continuous factor at level $m_j$ into
the paired discontinuous factor at the same level and therefore never
changes $\bm m$.  Hence the blockwise mapping immediately gives
$\dd V_{n,k}^l\subseteq V_{n,k}^{l+1}$.  The 1D commuting
relation \eqref{eq:1D-commutation} gives, for the level differences,
\[
  \partial_x\Delta_m^0=\Delta_m^1\partial_x,
  \qquad m\in\N_0.
\]
Together with the commutativity of operators acting in different
coordinates, this implies, for $j\notin I$,
\[
  \partial_{x_j}H_{\bm m}^{I}u_I
  =H_{\bm m}^{I\cup\{j\}}\partial_{x_j}u_I.
\]
Using the component formula for the exterior derivative and summing over
the unchanged index set $\mathcal M_n$, we conclude
\[
  \dd\mathcal I_n^l\omega
  =\sum_{\substack{I\subseteq[d]\\|I|=l}}
    \sum_{\bm m\in\mathcal M_n}
    \sum_{j\notin I}\epsilon(j,I)
    H_{\bm m}^{I\cup\{j\}}(\partial_{x_j}u_I)
    \,\dd x_{I\cup\{j\}}=\mathcal I_n^{l+1}\dd\omega.
\]
Thus the sparse spaces form a discrete de Rham subcomplex and the sparse
interpolants form a commuting family.

For comparison, let
\[
  \mathcal F_n
  :=
  \{\bm m\in\N_0^d:|\bm m|_\infty\le n\},
\]
and denote the corresponding full-grid finite element space of $l$-forms by
\begin{align*}
    V_{n,k}^{l,\mathrm{FG}}
  &:=
  \bigoplus_{I\subseteq[d],\,|I|=l}
  \left(
    \bigotimes_{i\in I}Q_n^k(x_i)
  \right)
  \otimes
  \left(
    \bigotimes_{j\notin I}S_n^{k+1}(x_j)
  \right)
  \dd x_I\\
  &=
  \bigoplus_{I\subseteq[d],\,|I|=l}
  \bigoplus_{\bm m\in\mathcal F_n}
  W_{\bm m}^{I,k}\,\dd x_I.
\end{align*}
The first expression is the global conforming finite element space obtained
by assembling the local shape function space \(\mathcal{Q}_{k+1}^{-}\Lambda^l\) \eqref{eq:trimmed-tensor-space} on the level-$n$ full cubical grid.  The
second expression is the same space written in the hierarchical basis.
Thus the sparse-grid space \(V_{n,k}^l\) is a subspace of the full-grid space 
\(V_{n+d-1,k}^{l,\mathrm{FG}}\). Their sizes are 
\[
  \dim V_{n,k}^l
  =O(2^n n^{d-1}),\qquad
  \dim V_{n,k}^{l,\mathrm{FG}}
  =O(2^{dn}),
\]
which is the basic sparse-grid complexity reduction
\citep{BungartzGriebel2004,GradinaruHiptmair2003}. In vector proxy, $V_{n,k}^{l,\mathrm{FG}}$ is nothing but the N\'ed\'elec $(l=1, d=3)$ or Raviart--Thomas $(l=2, d=3)$ finite element on the cubical mesh $\mathcal{T}_{(n,n,\ldots,n)}$.

\section{Sparse-grid approximation}
\label{sec:approximation}

The approximation error estimates in this section are stated for functions with mixed Sobolev
regularity.  For a component $u_I$, $|I|=l$, define
\[
  r_s=
  \begin{cases}
    k+1,&s\in I,\\
    k+2,&s\notin I,
  \end{cases}
\]
put $\bm r=(r_1,\ldots,r_d)$, and define the mixed Sobolev space
\[
  \mathcal X_I^k
  =
  \left\{
  v\in L^2(\Omega):
  \partial_{x_1}^{q_1}\cdots\partial_{x_d}^{q_d}v\in L^2(\Omega)
  \ \text{for every }\bm q\in\N_0^d
  \text{ with }\bm 0\le\bm q\le\bm r
  \right\}.
\]
This space is equipped
with the mixed norm
\begin{equation*}
  \norm[\mathcal X_I^k]{u_I}^2
  =
  \sum_{\bm q\in\N_0^d,\,\bm 0\le\bm q\le\bm r}
  \norm[L^2(\Omega)]{
    \partial_{x_1}^{q_1}\cdots\partial_{x_d}^{q_d}u_I}^2.
\end{equation*}
This norm is slightly stronger than necessary, but it cleanly controls all
mixed derivatives arising in the successive one-dimensional detail
estimates. Similar mixed Sobolev regularity is often called Korobov regularity and is used in neural-network approximation \cite{MontanelliDu2019,MaoZhou2022,LiZhang2025,LiZhang2026}.

For differential $l$-forms, set
\[
  \mathcal X_l^k
  =
  \left\{
  \omega=\sum_{|I|=l}u_I\,\dd x_I:
  u_I\in\mathcal X_I^k
  \right\},
  \qquad
  \norm[\mathcal X_l^k]{\omega}^2
  =\sum_{|I|=l}\norm[\mathcal X_I^k]{u_I}^2.
\]

Throughout this section, $C$ is independent of the grid level index $n$; its
dependence on $k$, $d$, or $l$ is indicated by subscripts.

\subsection{One-dimensional and tensor detail estimates}

For $m\ge1$, the difference
$\Delta_m^1=\Pi_m^1-\Pi_{m-1}^1$ is the $L^2$-orthogonal projection onto
$Z_m^{k}$.  On each cell $K\in\mathcal{T}_{m-1}$, since every detail element in $Z_K^k$ is orthogonal to $\mathbb{P}_k(K)$, we have, for any
$p_K\in\mathbb{P}_k(K)$,
\[
  \norm[L^2(K)]{\Delta_m^1q}
  =
  \norm[L^2(K)]{\Delta_m^1(q-p_K)}
  \le \norm[L^2(K)]{q-p_K}.
\]
The Bramble--Hilbert lemma, with $h_K=2^{-(m-1)}$, gives
\[
  \inf_{p_K\in\mathbb{P}_k(K)}\norm[L^2(K)]{q-p_K}
  \le C_kh_K^{k+1}|q|_{H^{k+1}(K)}.
\]
After squaring and summing over the parent cells, the fixed factor
$2^{k+1}$ is absorbed into $C_k$, and hence
\begin{equation}
  \norm[L^2(0,1)]{\Delta_m^1q}
  \le C_k2^{-(k+1)m}\norm[H^{k+1}(0,1)]{q}.
  \label{eq:DG-one-dimensional-decay}
\end{equation}
For the continuous detail $w=\Delta_m^0v$, the restriction $w|_K$ vanishes
at both endpoints of each parent cell.  The Poincar\'e inequality
and the 1D commuting relation give
\[
  \norm[L^2(K)]{w}
  \le C h_K\norm[L^2(K)]{w'}
  =
  C h_K\norm[L^2(K)]{\Delta_m^1v'}.
\]
Applying \eqref{eq:DG-one-dimensional-decay} to $v'$ and then summing over all cells in $\mathcal{T}_{m-1}$ therefore yields
\begin{equation}
  \norm[L^2(0,1)]{\Delta_m^0v}
  \le C_k2^{-(k+2)m}\norm[H^{k+2}(0,1)]{v}.
  \label{eq:continuous-1D-decay}
\end{equation}
At $m=0$, the exponential factors in
\eqref{eq:DG-one-dimensional-decay} and
\eqref{eq:continuous-1D-decay} equal one, and the
corresponding estimates follow from the boundedness of the fixed
finite-dimensional interpolation operators.

\begin{lemma}[Multidimensional detail estimate]
\label{thm:tensor-detail-estimate}
For $u_I\in\mathcal X_I^k$ and $\bm m\in\N_0^d$,
\begin{equation*}
  \norm[L^2(\Omega)]{H_{\bm m}^{I}u_I}
  \le C_{k,d}
  2^{-(k+1)\sum_{i\in I}m_i}
  2^{-(k+2)\sum_{j\notin I}m_j}
  \norm[\mathcal X_I^k]{u_I}.
\end{equation*}
\end{lemma}

\begin{proof}
Write $B_{s,m_s}=\Delta_{s,m_s}^1$ for $s\in I$ and
$B_{s,m_s}=\Delta_{s,m_s}^0$ for $s\notin I$, and let $r_s=k+1$ or $k+2$,
respectively.  Regard all variables except $x_1$ as parameters.  The
1D estimate  and Fubini's theorem
give
\[
\begin{aligned}
 \norm[L^2(\Omega)]{B_{1,m_1}\cdots B_{d,m_d}u_I}^2
 &\le C_k^2 2^{-2r_1m_1}
 \sum_{q_1=0}^{r_1}
 \norm[L^2(\Omega)]{
   \partial_{x_1}^{q_1}B_{2,m_2}\cdots B_{d,m_d}u_I}^2\\
 &=C_k^2 2^{-2r_1m_1}
 \sum_{q_1=0}^{r_1}
 \norm[L^2(\Omega)]{
   B_{2,m_2}\cdots B_{d,m_d}\partial_{x_1}^{q_1}u_I}^2.
\end{aligned}
\]
The equality uses the commuting property of operators acting in different coordinates.
Applying the same argument in $x_2,\ldots,x_d$ yields
\[
 \norm[L^2(\Omega)]{H_{\bm m}^{I}u_I}^2
 \le C_{k,d}^2 2^{-2\sum_{s=1}^d r_sm_s}
 \sum_{\substack{\bm q\in\N_0^d\\\bm 0\le\bm q\le\bm r}}
 \norm[L^2(\Omega)]{
   \partial_{x_1}^{q_1}\cdots\partial_{x_d}^{q_d}u_I}^2.
\]
The last sum is exactly $\norm[\mathcal X_I^k]{u_I}^2$, which completes the proof.
\end{proof}

\subsection{Summation of the sparse tail}

The full hierarchical expansion converges in $L^2$ for
$u_I\in\mathcal X_I^k$, and the sparse-grid error is exactly
\[
  u_I-\mathcal I_n^{I}u_I
  =
  \sum_{|\bm m|_1>n+d-1}H_{\bm m}^{I}u_I.
\]
Set $N=n+d-1$. For
fixed $I\subset[d]$, $|I|=l$, the scalar tail that remains after applying
Lemma~\ref{thm:tensor-detail-estimate} is
\[
  S_I(N)
  :=
  \sum_{|\bm m|_1>N}
  2^{-(k+1)\sum_{i\in I}m_i}
  2^{-(k+2)\sum_{j\notin I}m_j}.
\]
We estimate $S_I(N)$ according to the form degree.
We use the exponentially weighted tail estimate
\begin{equation}
  \sum_{s>N}2^{-\sigma s}(s+1)^q
  \le C_{\sigma,q}2^{-\sigma N}(N+1)^q,
  \qquad \sigma>0,\quad q\in\N_0.
  \label{eq:exponentially-weighted-tail}
\end{equation}
Indeed, writing $s=N+i$ and using
$N+i+1\le(N+1)(i+1)$ gives
\[
  \sum_{s>N}2^{-\sigma s}(s+1)^q
  \le
  2^{-\sigma N}(N+1)^q
  \sum_{i=1}^{\infty}2^{-\sigma i}(i+1)^q,
\]
and the last series is finite for $\sigma>0$.

If $l=0$, then $I=\varnothing$, so every direction has exponent $k+2$.
Decomposing the summation domain into the shells
$\{\bm m\in\N_0^d:|\bm m|_1=s\}$ gives
\begin{align*}
  S_\varnothing(N)
  &=\sum_{s>N}
    \sum_{\substack{\bm m\in\N_0^d\\|\bm m|_1=s}}
    2^{-(k+2)|\bm m|_1} \\
  &=\sum_{s>N}
    \#\{\bm m\in\N_0^d:|\bm m|_1=s\}\,2^{-(k+2)s} \\
  &=\sum_{s>N}\binom{s+d-1}{d-1}2^{-(k+2)s} \\
  &\le \sum_{s>N}2^{-(k+2)s}(s+1)^{d-1} \\
  &\le C_{k,d}2^{-(k+2)N}(N+1)^{d-1},
\end{align*}
where the last step follows from
\eqref{eq:exponentially-weighted-tail} with $\sigma=k+2$ and $q=d-1$.

Suppose next that $1\le l<d$.  For each $\bm m\in\N_0^d$, set
\[
  a(\bm m)=\sum_{i\in I}m_i,
  \qquad
  b(\bm m)=\sum_{j\notin I}m_j,
  \qquad
  |\bm m|_1=a(\bm m)+b(\bm m).
\]
For fixed $s\in\N_0$ and $0\le t\le s$,
\begin{align*}
\#\{\bm m\in\N_0^d:a(\bm m)=s-t,\ b(\bm m)=t\}&=
  \binom{s-t+l-1}{l-1}
  \binom{t+d-l-1}{d-l-1},\\
  \binom{s-t+l-1}{l-1}\le (s+1)^{l-1},&
  \qquad
  \sum_{t=0}^{\infty}2^{-t}
  \binom{t+d-l-1}{d-l-1}= 2^{d-l}.
\end{align*}

Summing over the shell index $s=|\bm m|_1$ and $b(\bm m)=t$, we obtain
\begin{align*}
  S_I(N)
  &=\sum_{s>N}
    \sum_{\substack{\bm m\in\N_0^d\\|\bm m|_1=s}}
    2^{-(k+1)s}2^{-b(\bm m)} \\
  &=\sum_{s>N}2^{-(k+1)s}
    \sum_{t=0}^{s}2^{-t}
    \#\{\bm m\in\N_0^d:a(\bm m)=s-t,\ b(\bm m)=t\}\\
  &=\sum_{s>N}2^{-(k+1)s}
    \sum_{t=0}^{s}2^{-t}
    \binom{s-t+l-1}{l-1}
    \binom{t+d-l-1}{d-l-1} \\
  &\le 2^{d-l}\sum_{s>N}2^{-(k+1)s}(s+1)^{l-1} \\
  &\le C_{k,d,l}2^{-(k+1)N}(N+1)^{l-1}.
\end{align*}
The last step is \eqref{eq:exponentially-weighted-tail} with
$\sigma=k+1$ and $q=l-1$.

Finally, if $l=d$, there are no transverse directions.
Hence
\begin{align*}
  S_{[d]}(N)=\sum_{s>N}\binom{s+d-1}{d-1}2^{-(k+1)s}\le C_{k,d}2^{-(k+1)N}(N+1)^{d-1},
\end{align*}

Since $2^{-(k+1)N}=2^{-(k+1)n}2^{-(k+1)(d-1)}$ and $(N+1)^{l-1}\leq d^{l-1}(n+1)^{l-1}$, replacing $N$ by $n$ in these bounds changes
only constants depending on $k$, $d$, and $l$.

\begin{theorem}[Sparse-grid $L^2$ approximation]
\label{thm:L2-approximation}
Let $\omega=\sum_{|I|=l}u_I\dd x_I$ with each coefficient
$u_I\in\mathcal X_I^k$ and
\[  E_{l,k,d}(n)
  =
  \begin{cases}
    2^{-(k+2)n}(n+1)^{d-1},&l=0,\\
    2^{-(k+1)n}(n+1)^{l-1},&1\le l\le d.
  \end{cases}\]
Then there is a constant $C_{k,d,l}$ independent of
$n$ such that
\begin{align*}
  \norm[L^2\Lambda^l(\Omega)]{
    \omega-\mathcal I_n^l\omega}
  &\le C_{k,d,l}E_{l,k,d}(n)
  \norm[\mathcal X_l^k]{\omega}.
\end{align*}
\end{theorem}

\begin{proof}
By the triangle inequality and Lemma \ref{thm:tensor-detail-estimate}, each
component satisfies
\[
  \norm[L^2(\Omega)]{u_I-\mathcal I_n^{I}u_I}
  \le C_{k,d}\norm[\mathcal X_I^k]{u_I}
  \sum_{|\bm m|_1>N}
  2^{-(k+1)\sum_{i\in I}m_i}
  2^{-(k+2)\sum_{j\notin I}m_j}.
\]
The preceding casewise estimates bound this sum by
$C_{k,d,l}E_{l,k,d}(n)$.  Since
the coordinate forms $\{\dd x_I\}$ are orthonormal,
\[
  \norm[L^2\Lambda^l(\Omega)]{\omega-\mathcal I_n^l\omega}^2
  =
  \sum_{|I|=l}
  \norm[L^2(\Omega)]{u_I-\mathcal I_n^Iu_I}^2.
\]
Substitution of the component estimates and summation over the fixed number
$\binom dl$ of components proves the theorem.
\end{proof}

For \(0\le l\le d-1\), the sparse-grid approximation error bound in the graph follows from the commuting property.  In fact,
\[
  \dd(\omega-\mathcal I_n^l\omega)
  =
  \dd\omega-\mathcal I_n^{l+1}\dd\omega.
\]
Applying Theorem \ref{thm:L2-approximation} once to $\omega$ and once to
$\dd\omega$ proves the following
corollary.

\begin{corollary}\label{cor:graph-approximation}
Let \(0\le l\le d-1\).  If $\omega\in\mathcal X_l^k$ and
$\dd\omega\in\mathcal X_{l+1}^k$, then
\begin{equation*}
  \norm[H(\dd,\Omega)]{\omega-\mathcal I_n^l\omega}
  \le C_{k,d,l}\Big(E_{l,k,d}(n)\norm[\mathcal X_l^k]{\omega} +E_{l+1,k,d}(n)
  \norm[\mathcal X_{l+1}^k]{\dd\omega}\Big).
\end{equation*}
\end{corollary}

Let $h_n=2^{-n}$ be the mesh-size of the $n$th grid level.
For a sufficiently regular $l$-form $\omega$, the $H(\dd)$ error is therefore governed by
$O(h_n^{k+1}|\log h_n|^l)$.  Since the dimension of sparse-grid space $V_{n,k}^l$
is only $O(h_n^{-1}|\log h_n|^{d-1})$, the convergence rate of higher-order sparse-grid methods for differential forms is
approximately $O(N_{\rm dof}^{-k-1})$ modulo logarithmic factor, where $N_{\rm dof}$ is the number of DoFs. In contrast, a full-grid tensor-product finite element method has a convergence rate 
approximately $N_{\rm dof}^{-(k+1)/d}$.

\section{Exactness and discrete stable potentials}
\label{sec:stability}

Because the cube $\Omega$ is contractible, the de Rham complex
\begin{equation*}
  \R\longrightarrow H\Lambda^0(\Omega)
  \xrightarrow{\dd}H\Lambda^1(\Omega)
  \xrightarrow{\dd}\cdots
  \xrightarrow{\dd}H\Lambda^d(\Omega)
  \longrightarrow0
\end{equation*}
is exact. In this section, we construct an explicit homotopy that preserves exactness at the sparse-grid level.

\subsection{A tensor-product homotopy}

On $(0,1)$, for $v,q\in L^2(0,1)$, we define two operators $p$ and $h$ acting on the scalar $0$-form
$v$ and the $1$-form $q\,\dd x$ by
\begin{align*}
    pv&=\left(\int_0^1v(t)\,\dd t\right)1,\qquad
  p(q\,\dd x)=0\,\dd x,\\
hv&=0,\qquad h(q\,\dd x)=Aq:=
  \int_0^xq(t)\,\dd t
  -
  \int_0^1\int_0^yq(t)\,\dd t\,\dd y.
\end{align*}
Here $Aq$ is the primitive of $q$ that satisfies $(Aq)'=q$, $\int_0^1Aq=0$, and a Poincar\'e
inequality gives
\begin{equation}
  \norm[L^2(0,1)]{Aq}
  \le\frac1\pi\norm[L^2(0,1)]{q}.
  \label{eq:centered-primitive-bound}
\end{equation}
For $v\in H^1(0,1)$, the definition of $A$ implies $A(v')
  =v-\left(\int_0^1v(t)\,\dd t\right)1$,
so a direct calculation on the 1D complex
$H^1(0,1)\xrightarrow{\dd}L^2(0,1)\,\dd x$ gives
\begin{equation}\label{eq:1D-homotopy}
  \dd h+h\dd=\id-p,
  \qquad
  \dd p=p\dd=0.
\end{equation}

We now tensorize this 1D construction to construct a multidimensional stable homotopy operator, which can be viewed as an instance of
the tensor trick in algebraic topology
\cite{GugenheimLambeStasheff1991,Berglund2014}, although the derivations there are rather abstract and have not been used in numerical analysis before. Other applications of algebraic-topological techniques in numerical analysis can be found in \cite{ChristiansenMuntheKaasOwren2011}, where the exactness of tensor-product finite element complexes was proved using the K{\"u}nneth theorem.

Let $p_i$ and $h_i$ denote $p$ and $h$ acting in
coordinate $x_i$.  Their
action on a component $u_I\dd x_I$ is explicit. In particular,
\begin{equation}
  p_i(u_I\dd x_I)
  =
  \begin{cases}
    \Big(\int_0^1u_I(x_1,\ldots,t,\ldots,x_d)\,\dd t\Big)\dd x_I,&i\notin I,\\
    0,&i\in I.
  \end{cases}
  \label{eq:coordinate-mean}
\end{equation}
If $i=i_r$ is the $r$th element of the ordered set $I$, let $A_i$ be the operator $A$ acting only in the $x_i$ coordinate and set
\begin{equation}
  h_i(u_I\dd x_I)
  =
  \begin{cases}
    0,&i\notin I,\\
    (-1)^{r-1}(A_i u_I)\dd x_{I\setminus\{i\}},&i=i_r\in I.
  \end{cases}
  \label{eq:coordinate-homotopy}
\end{equation}
The sign in \eqref{eq:coordinate-homotopy} is the Koszul sign generated by
moving a degree-lowering operator past the preceding one-form factors. We further define
\begin{equation*}
    P_{<1}:=\id,\qquad
    P_{<i}:=p_1\cdots p_{i-1},\qquad \mathcal{H}:=\sum_{i=1}^dP_{<i}h_i.
\end{equation*}
In addition, we consider the operator
\[
\mathcal{P}:=p_1p_2\cdots p_d,
\]
which vanishes on $l$-forms with $l\geq1$, while
$\mathcal{P}v=(\int_\Omega v\,\dd\bm x)1$ for a 0-form $v$.
The next lemma is the key tool for proving discrete exactness.
\begin{lemma}[Multidimensional homotopy]\label{lem:homotopy}
 For smooth $l$-forms with $0\leq l\leq d$, it holds that
  \[
  \dd\mathcal{H} + \mathcal{H}\dd=\id-\mathcal{P}.
  \]
\end{lemma}
\begin{proof}
Write $\dd=\sum_{j=1}^d\dd_j$, where $\dd_j$ differentiates only in the
$x_j$ coordinate.  The 1D identity \eqref{eq:1D-homotopy} gives
\begin{subequations}
  \begin{align}
      \dd_i h_i+h_i\dd_i&=\id-p_i,\label{eq:xi-homotopy}\\
  \dd_i p_i&=p_i\dd_i=0.
  \end{align}
\end{subequations}
We claim
from \eqref{eq:coordinate-homotopy} that
\begin{equation}\label{eq:distinct_id}
  \dd_i h_j+h_j\dd_i=0,
  \qquad i\ne j.
\end{equation}
Indeed, both compositions vanish if $j\notin I$ or $i\in I$. It remains to
consider $j\in I$ and $i\notin I$. Let $u_I\dd x_I$ be a smooth
differential form. Using the orientation sign $\epsilon$ defined in
Section~\ref{sec:sparse-interpolation},
\begin{align*}
  \dd_i h_j(u_I\dd x_I)
  &={}
  \epsilon(j,I)\epsilon(i,I\setminus\{j\})
  (\partial_{x_i}A_j u_I)
  \dd x_{(I\setminus\{j\})\cup\{i\}},\\
  h_j\dd_i(u_I\dd x_I)
  &={}
  \epsilon(i,I)\epsilon(j,I\cup\{i\})
  (A_j\partial_{x_i}u_I)
  \dd x_{(I\cup\{i\})\setminus\{j\}}.
\end{align*}
The two ordered index sets are equal, and differentiation in $x_i$
commutes with integration in $x_j$, so
$\partial_{x_i}A_j u_I=A_j\partial_{x_i}u_I$.  Finally,
\[
  \epsilon(j,I)\epsilon(i,I\setminus\{j\})
  =-\epsilon(i,I)\epsilon(j,I\cup\{i\}),
\]
because inserting $\dd x_i$ and removing $\dd x_j$ in the opposite order
changes the orientation by one transposition.  Hence the two displayed
terms cancel.

We may now compute in arbitrary dimension.  The operators $\dd_j$ and
$h_j$ commute with $p_r$ whenever $r\ne j$.  Therefore
\begin{align*}
  \dd\mathcal{H}+\mathcal{H}\dd
  &={}
  \sum_{i=1}^d\sum_{j=1}^d
  \bigl(\dd_jP_{<i}h_i+P_{<i}h_i\dd_j\bigr).
\end{align*}
For $j>i$, the two terms cancel by the identity \eqref{eq:distinct_id}.  For $j<i$, the product $P_{<i}$ contains $p_j$, and
$\dd_jp_j=p_j\dd_j=0$, so both terms vanish. Therefore, only summands with $j=i$ remain. Using \eqref{eq:xi-homotopy} yields
\begin{align*}
  \dd\mathcal{H}+\mathcal{H}\dd
  &=\sum_{i=1}^dP_{<i}(\dd_i h_i+h_i\dd_i)=\sum_{i=1}^dP_{<i}(\id-p_i)\\
  &=(\id-p_1)+p_1(\id-p_2)+\cdots
    +p_1\cdots p_{d-1}(\id-p_d)\\
  &=\id-p_1\cdots p_d
  =\id-\mathcal{P}.
\end{align*}
The proof is complete.
\end{proof}

\begin{lemma}[Boundedness of homotopy operator]\label{lem:bounded_homotopy}
  For every $\omega\in L^2\Lambda^l(\Omega)$,
\begin{equation*}
  \norm[L^2\Lambda^{l-1}(\Omega)]{\mathcal{H}\omega}
  \le\frac1\pi
  \norm[L^2\Lambda^l(\Omega)]{\omega},
  \qquad 1\le l\le d.
\end{equation*}
\end{lemma}
\begin{proof}
Let $\omega=\sum_{I\subseteq[d],\, |I|=l}u_I\dd x_I$ be an $l$-form. Denote the $i$th summand of \(\mathcal{H}\) by $\mathcal{H}_i:=P_{<i}h_i$. The component formulas \eqref{eq:coordinate-mean} and
\eqref{eq:coordinate-homotopy} imply
\[
  \mathcal{H}_i(u_I\dd x_I)=0
  \qquad\text{for } i\neq\min I.
\]
Indeed, \(i\notin I\) gives \(h_i(u_I\dd x_I)=0\).  If \(i\in I\) but
\(j:=\min I<i\), then \(P_{<i}\) contains \(p_j\), and \(p_j\) annihilates
the one-form factor \(\dd x_j\).
Consequently,
\[
T_I:=\mathcal{H}_i(u_I\dd x_I)
      =(P_{<i}A_i u_I)\dd x_{\check{I}}\qquad
  \text{for } i=\min I,
\]
where \(\check{I}:=I\setminus\{\min I\}\). By the Cauchy--Schwarz inequality in \(x_j\) and Fubini's theorem, each
\(p_j\) satisfies \(\norm[L^2(\Omega)]{p_j\omega}\le\norm[L^2(\Omega)]{\omega}\) and thus $\norm[L^2(\Omega)]{P_{<i}\omega}\le\norm[L^2(\Omega)]{\omega}$.  Applying
\eqref{eq:centered-primitive-bound} in \(x_i\) then gives
\begin{equation}\label{eq:TIbound}
    \norm[L^2(\Omega)]{T_I}
  \le \norm[L^2(\Omega)]{A_i u_I}
  \le \frac1\pi\norm[L^2(\Omega)]{u_I}.
\end{equation}

We next show \(\{T_I:|I|=l\}\) are pairwise
orthogonal.  First,
\[
  (T_I,T_J)_{L^2\Lambda^{l-1}}=0\quad\text{if }\check{I}\ne \check{J},
\]
because \(\dd x_{\check{I}}\) and \(\dd x_{\check{J}}\) are defined to be
orthogonal. Then we
consider \(\check{I}=\check{J}\) while $I\neq J$.  Write \(i=\min I\), \(j=\min J\neq i\), and assume \(i<j\)
without loss of generality. Since \(A_i\) is the zero mean-value primitive in $x_i$, we have
\[
  \int_0^1 P_{<i}A_i u_I\,\dd x_i=0.
\]
On the other hand, \(P_{<j}\) contains \(p_i\), so the coefficient of
\(T_J\) is independent of \(x_i\). Let
$\dd\bm x_{\hat{i}}:=\prod_{s\in[d]\setminus\{i\}}\dd x_s$.
Fubini's theorem therefore yields
\[
  (T_I,T_J)_{L^2\Lambda^{l-1}}
  =
  \int_{(0,1)^{d-1}}
  \left[
    \int_0^1
    (P_{<i}A_i u_I)(P_{<j}A_j u_J)\,\dd x_i
  \right]\dd\bm x_{\hat{i}}
  =0.
\]
Thus the family \(\{T_I:|I|=l\}\) is pairwise orthogonal.

It then follows from the orthogonality and \eqref{eq:TIbound} that
\[
  \norm[L^2\Lambda^{l-1}(\Omega)]{\mathcal{H}\omega}^2
  =
  \sum_{|I|=l}\norm[L^2(\Omega)]{T_I}^2
  \le
  \frac1{\pi^2}\sum_{|I|=l}\norm[L^2(\Omega)]{u_I}^2
  =
  \frac1{\pi^2}\norm[L^2\Lambda^l(\Omega)]{\omega}^2.
\]
The proof is complete.
\end{proof}
Smooth forms are dense in $H\Lambda^l(\Omega)$.  The $L^2$-boundedness of
$\mathcal{H}$ and $\mathcal{P}$, together with the closedness of $\dd$, therefore extends
the homotopy identity to $H\Lambda^l(\Omega)$. More precisely,
\begin{align*}
  \mathcal{H}\omega&\in H\Lambda^{l-1}(\Omega), &&1\le l\le d,\\
  \dd\mathcal{H}\omega+\mathcal{H}\dd\omega
  &=\omega-\mathcal{P}\omega, &&0\le l\le d,
  \quad \omega\in H\Lambda^l(\Omega).
\end{align*}

\subsection{Exactness and stable potentials}

At lowest order $(k=0)$, \citet{GradinaruHiptmair2003} proved the terminal
stable-potential estimate from degree $d-1$ to degree $d$, while
\cite{Gradinaru2002} treated several additional form degrees, but did not
establish the stable-potential estimate for all $l$ in arbitrary dimension, even in the lowest order case $k=0$.

With the help of the homotopy operator $\mathcal{H}$, we are able to establish the existence of stable discrete potentials for arbitrary higher-order sparse-grid forms in arbitrary dimension. The stability is robust with respect to the polynomial degree. To this end, we need to show that $\mathcal{H}$ preserves sparse-grid differential forms.

\begin{lemma}\label{lem:sparse-homotopy-preservation}
  For $1\leq l\leq d$, we have
  \begin{equation*}
  \mathcal{H} V_{n,k}^l\subseteq V_{n,k}^{l-1}.
\end{equation*}
\end{lemma}
\begin{proof}
  We start with the actions of 1D operators $p$ and $h$:
\begin{equation} \label{eq:1D-sparse-preservation}
  p\bigl(Y_m^{k+1}\bigr)\subseteq Y_0^{k+1},\qquad
  h\bigl(Z_m^{k}\,\dd x\bigr)
  \subseteq Y_m^{k+1}+Y_0^{k+1},
\end{equation}
because $p$ returns a constant and $h$ returns an antiderivative.
Recall $Y_m^{k+1}$ is a scalar continuous 1D detail space of $0$-forms, whereas
$Z_m^{k}\dd x$ is the discontinuous 1D detail space of $1$-forms.

Recall the decomposition
$\mathcal{H}=\sum_{i=1}^d\mathcal{H}_i$ with $\mathcal{H}_i=P_{<i}h_i$ introduced in the proof of Lemma \ref{lem:homotopy}. Let $I\subseteq[d]$ be an ordered index set with $|I|=l$.  Applying $\mathcal{H}_i$ to the multidimensional detail space
$W_{\bm m}^{I,k}\dd x_I$ in \eqref{eq:WmI} and using the
formulas \eqref{eq:coordinate-mean} and
\eqref{eq:coordinate-homotopy} gives
\[
  \mathcal{H}_i\bigl(W_{\bm m}^{I,k}\dd x_I\bigr)=0
  \qquad\text{for }
  i\neq\min I.
\]
For the grid level $\bm{m}=(m_1,\ldots,m_d)$ and \(i=\min I\), define
\[
  \mathcal R_i(\bm m)
  :=
  \left\{
    \bm r\in\N_0^d: r_i\in\{0,m_i\},\, r_s=0\text{ for }s<i,\, r_s=m_s \text{ for }s>i \right\}.
\]
The 1D inclusions \eqref{eq:1D-sparse-preservation}
then yield the blockwise relation
\[
  \mathcal{H}_i\bigl(W_{\bm m}^{I,k}\dd x_I\bigr)
  \subseteq
  \bigoplus_{\bm r\in\mathcal R_i(\bm m)}
  W_{\bm r}^{\check{I},k}\dd x_{\check{I}},
  \qquad i=\min I.
\]
Here the coordinates \(s<i\) are sent to level zero by \(p_s\), the
\(i\)th factor is sent by \(h_i\) to level \(m_i\) or level zero, and the
remaining factors are unchanged.  In particular,
\[
  \bm r\in\mathcal R_i(\bm m)
  \quad\Longrightarrow\quad
  \bm 0\le\bm r\le\bm m.
\]
By definition, \(\mathcal{M}_n\) is downward closed, namely,
\[
  \bm m\in\mathcal{M}_n,\quad \bm 0\le\bm r\le\bm m
  \quad\Longrightarrow\quad
  \bm r\in\mathcal{M}_n.
\]
Summing over all form indices \(I\) and all \(\bm m\in\mathcal{M}_n\) completes the proof.
\end{proof}

Define the sparse-grid cocycle and coboundary spaces
\[
\mathcal{Z}_{n,k}^l=\ker(\dd|_{V_{n,k}^l}),\qquad\mathcal{B}_{n,k}^l=\dd V_{n,k}^{l-1}.
\]
The cochain property $\dd^2=0$ implies $\mathcal{B}_{n,k}^l\subseteq\mathcal{Z}_{n,k}^l$.
We present the main result on the existence of stable discrete potentials in the next theorem.
\begin{theorem}[Stable discrete potential]
\label{thm:stable-potential}
For $1\le l\le d$ and every $v_h\in\mathcal{Z}_{n,k}^l$, the form
$w_h=\mathcal{H} v_h$ is a member of $V_{n,k}^{l-1}$ that satisfies
\begin{equation*}
  \dd w_h=v_h,\qquad
  \norm[L^2(\Omega)]{w_h}\le\frac1\pi\norm[L^2(\Omega)]{v_h}.
\end{equation*}
In particular, the sparse-grid complex $(V_{n,k}^\bullet,\dd^\bullet)$ is exact:
\begin{equation*}
  \ker(\dd|_{V_{n,k}^0})=\R,\qquad
  \mathcal{Z}_{n,k}^l=\mathcal{B}_{n,k}^l,\qquad 1\le l\le d.
\end{equation*}
\end{theorem}

\begin{proof}
  Since $l\ge1$, $\mathcal{P} v_h=0$. The homotopy in Lemma \ref{lem:homotopy} implies
\[
  v_h=(\dd\mathcal{H}+\mathcal{H}\dd)v_h=\dd\mathcal{H} v_h.
\]
The $L^2$ boundedness of $w_h=\mathcal{H} v_h$ follows from
Lemma \ref{lem:bounded_homotopy}.  The equality $\mathcal{B}_{n,k}^l=\mathcal{Z}_{n,k}^l$ then follows from $\mathcal{B}_{n,k}^l\subseteq\mathcal{Z}_{n,k}^l$, and the constructed potential.

The identity $\ker(\dd|_{V_{n,k}^0})=\R$ follows because $\dd$ maps $0$-forms to $1$-forms as the gradient operator.
\end{proof}

Consider the discrete Hodge decomposition
\[
V_{n,k}^l=\mathcal{Z}_{n,k}^l\oplus_{L^2}\mathcal{Z}_{n,k}^{l,\perp}.
\]
Theorem \ref{thm:stable-potential} implies the following discrete Poincar\'e inequality:
\begin{equation}\label{eq:Poincare}
  \|w_h\|_{L^2(\Omega)}\leq \frac{1}{\pi}\|\dd w_h\|_{L^2(\Omega)},\qquad\forall w_h\in\mathcal{Z}_{n,k}^{l,\perp}.
\end{equation}

A distinct feature of our result is that the stability constant in Theorem \ref{thm:stable-potential} and \eqref{eq:Poincare} is $1/\pi$, and thus uniform with respect to polynomial degree $k$. The same analysis leads to existence of stable and polynomial-degree-robust discrete potential on full tensor-product grids. 
\begin{corollary}
For $1\le l\le d$ and every $v_h\in V_{n,k}^{l,\rm FG}$ with $\dd v_h=0$, the form
$w_h=\mathcal{H} v_h$ is a member of $V_{n,k}^{l-1,\rm FG}$ that satisfies
\begin{equation*}
  \dd w_h=v_h,\qquad
  \norm[L^2(\Omega)]{w_h}\le\frac1\pi\norm[L^2(\Omega)]{v_h}.
\end{equation*}
\end{corollary}
In comparison, classical analysis for uniform bound of discrete potential on general triangular grids is not constructive \cite{Boffi2011}.

\subsection{Bounded cochain projection}
As shown in \citep[Theorem~3.7]{ArnoldFalkWinther2010}, a bounded cochain projection follows from the existence of bounded discrete potentials in Theorem \ref{thm:stable-potential}.

For convenience,
operators and spaces with superscript $-1$ or $d+1$ are set to zero.
By the definition of
$\mathcal{B}_{n,k}^{l+1}$, the restriction of
$\dd|_{\mathcal{Z}_{n,k}^{l,\perp}}: \mathcal{Z}_{n,k}^{l,\perp}\rightarrow\mathcal{B}_{n,k}^{l+1}$ is a bijection and we set
$R_{n,k}^{l+1}:=(\dd|_{\mathcal{Z}_{n,k}^{l,\perp}})^{-1}$.
Theorem \ref{thm:stable-potential} or the inequality \eqref{eq:Poincare} simply says that
\begin{equation}\label{eq:right-inverse-bound}
  \norm[L^2(\Omega)]{R_{n,k}^{l+1}b_h}
  \le\frac1\pi\norm[L^2(\Omega)]{b_h},\qquad\forall\, b_h\in\mathcal{B}_{n,k}^{l+1}.
\end{equation}

For any discrete subspace $X$, let $P_X$ denote the $L^2$-orthogonal
projection onto $X$. We shall show that
\begin{equation}\label{eq:cochain-projection}
  \pi_{n,k}^l\omega
  =
  P_{\mathcal{Z}_{n,k}^l}\omega+
  R_{n,k}^{l+1}P_{\mathcal{B}_{n,k}^{l+1}}\dd\omega,\qquad\omega\in H\Lambda^l(\Omega),
\end{equation}
is a commuting projection bounded in the $H(\dd)$ norm.

\begin{theorem}[Bounded commuting projection]
\label{thm:bounded-commuting-projection}
For the sparse-grid complex, the projections
$\pi_{n,k}^l:H\Lambda^l(\Omega)\to V_{n,k}^l$ in \eqref{eq:cochain-projection} satisfy
  \begin{align*}
      \dd\pi_{n,k}^l&=\pi_{n,k}^{l+1}\dd,\qquad 0\le l<d,\\
  \norm[H(\dd)]{\pi_{n,k}^l\omega}
  &\le\sqrt{1+\pi^{-2}}\norm[H(\dd)]{\omega}.
  \end{align*}
\end{theorem}
\begin{proof}
  We first verify the projection property. If
$v_h\in V_{n,k}^l$, then $\dd v_h\in\mathcal{B}_{n,k}^{l+1}$ and hence
$P_{\mathcal{B}_{n,k}^{l+1}}\dd v_h=\dd v_h$. The term $v_h-P_{\mathcal{Z}_{n,k}^l}v_h$ is contained in $\mathcal{Z}_{n,k}^{l,\perp}$ and $\dd(v_h-P_{\mathcal{Z}_{n,k}^l}v_h)=\dd v_h$ because
$P_{\mathcal{Z}_{n,k}^l}v_h$ is closed.
Therefore $v_h-P_{\mathcal{Z}_{n,k}^l}v_h$ is the unique element of $\mathcal{Z}_{n,k}^{l,\perp}$ selected by the right
inverse:
\[
  R_{n,k}^{l+1}\dd v_h=v_h-P_{\mathcal{Z}_{n,k}^l}v_h.
\]
Substitution into \eqref{eq:cochain-projection} gives
$\pi_{n,k}^lv_h=v_h$ for all $v_h\in V_{n,k}^l$.

Next, $P_{\mathcal{Z}_{n,k}^l}\omega$ is closed, while
$\dd R_{n,k}^{l+1}$ is the identity on $\mathcal{B}_{n,k}^{l+1}$.  Therefore
\begin{equation}
  \dd\pi_{n,k}^l\omega=P_{\mathcal{B}_{n,k}^{l+1}}\dd\omega.
  \label{eq:derivative-of-cochain-projection}
\end{equation}
In what follows, the discrete exactness
$\mathcal{Z}_{n,k}^{l+1}=\mathcal{B}_{n,k}^{l+1}$ implies
\[
  \pi_{n,k}^{l+1}\dd\omega
  =
  P_{\mathcal{Z}_{n,k}^{l+1}}\dd\omega+
  R_{n,k}^{l+2}P_{\mathcal{B}_{n,k}^{l+2}}\dd^2\omega
  =
  P_{\mathcal{B}_{n,k}^{l+1}}\dd\omega.
\]
Comparison with \eqref{eq:derivative-of-cochain-projection} proves $\dd\pi_{n,k}^l=\pi_{n,k}^{l+1}\dd$.

The Pythagorean theorem and
\eqref{eq:right-inverse-bound} yield
\[
\begin{aligned}
  \norm[L^2(\Omega)]{\pi_{n,k}^l\omega}^2
  &=
  \norm[L^2(\Omega)]{P_{\mathcal{Z}_{n,k}^l}\omega}^2+
  \norm[L^2(\Omega)]{R_{n,k}^{l+1}P_{\mathcal{B}_{n,k}^{l+1}}\dd\omega}^2\\
  &\le
  \norm[L^2(\Omega)]{\omega}^2+
  \pi^{-2}\norm[L^2(\Omega)]{\dd\omega}^2.
\end{aligned}
\]
Equation \eqref{eq:derivative-of-cochain-projection} also yields
\[
  \norm[L^2(\Omega)]{\dd\pi_{n,k}^l\omega}
  =
  \norm[L^2(\Omega)]{P_{\mathcal{B}_{n,k}^{l+1}}\dd\omega}
  \le\norm[L^2(\Omega)]{\dd\omega}.
\]
The proof is complete.
\end{proof}

Equivalently, the continuous and sparse discrete complexes are connected by
the following commuting projection diagram:
\[
\begin{tikzcd}[column sep=small,row sep=large]
  \R
    \arrow[r,hook]
    \arrow[d,"\id"']
  & H\Lambda^0(\Omega)
    \arrow[r,"\dd"]
    \arrow[d,"\pi_{n,k}^0"']
  & H\Lambda^1(\Omega)
    \arrow[r,"\dd"]
    \arrow[d,"\pi_{n,k}^1"']
  & \cdots
    \arrow[r,"\dd"]
  & H\Lambda^d(\Omega)
    \arrow[r]
    \arrow[d,"\pi_{n,k}^d"']
  & 0
  \\
  \R
    \arrow[r,hook]
  & V_{n,k}^0
    \arrow[r,"\dd"]
  & V_{n,k}^1
    \arrow[r,"\dd"]
  & \cdots
    \arrow[r,"\dd"]
  & V_{n,k}^d
    \arrow[r]
  & 0.
\end{tikzcd}
\]

By standard FEEC theory, Theorem \ref{thm:bounded-commuting-projection}
yields quasi-optimal convergence for sparse-grid mixed discretizations of
Hodge--Laplacian source problems, including mixed Poisson and magnetostatics.

\section{Numerical experiments}
\label{sec:numerics}

The experiments use the sparse-grid spaces \(V_{n,k}^1\) of 1-forms and their
full-grid counterparts \(V_{n,k}^{1,\mathrm{FG}}\) defined in
Section~\ref{sec:sparse-interpolation}, with \(k=0,1,2\). As is common in FEEC,
we use the identification between 1-forms and vector fields
\begin{equation*}
  \begin{aligned}
    \omega&=\sum_{i=1}^d u_i\,\dd x_i
    \ \longleftrightarrow\ \bm u=(u_1,\ldots,u_d),\\
    \dd\omega&\ \longleftrightarrow\ \curl\bm u=\left\{\begin{aligned}
      (\partial_{x_2}u_3-\partial_{x_3}u_2,\partial_{x_3}u_1-\partial_{x_1}u_3,\partial_{x_1}u_2-\partial_{x_2}u_1)\quad\text{in 3D},\\
      \partial_{x_1}u_2-\partial_{x_2}u_1\quad\text{in 2D}.
    \end{aligned}\right.
  \end{aligned}
\end{equation*}

Let $H_0(\curl;\Omega)$ be the subspace with essential boundary condition:
\[
H_0(\curl;\Omega)=\{\bm{v}\in L^2(\Omega): \curl\bm{v}\in L^2(\Omega),\, \bm{v}\times\bm{n}=\bm{0}\text{ on }\partial\Omega\},
\]
where $\bm{n}$ is the outward unit normal to $\partial\Omega$. We use
\[
  V_h^1
  =
  \begin{cases}
    V_{n,k}^1\cap H_0(\curl;\Omega),
      &\text{for a sparse-grid experiment},\\
    V_{n,k}^{1,\mathrm{FG}}\cap H_0(\curl;\Omega),
      &\text{for a full-grid experiment}.
  \end{cases}
\]
To treat the essential boundary condition, \(S_m^{k+1}\) is restricted
to \(S_m^{k+1}\cap H_0^1(0,1)\) and the discontinuous space \(Q_m^k\) is unchanged. In the crucial 1D hierarchical basis, all details in \(Y_m^{k+1}\) with \(m\ge1\) already vanish at the endpoints. Therefore,
only the level-zero basis \(\theta_{0,\star}^{k+1}=1\) and
\(\theta_{0,0}^{k+1}=x\) are dropped. 

We shall record the fitted order $r$ of convergence such that the $H(\rm curl)$ error is approximately $O(N_{\rm dof}^{-r})$, where $N_{\rm dof}$ is the number of DoFs in sparse- or full-grid methods.

\subsection{Two- and three-dimensional source problems}

We solve
\begin{equation*}
  \curl\curl\bm u+\bm u=\bm f\quad\text{in }\Omega,
  \qquad
  \bm u\times\bm n=0\quad\text{on }\partial\Omega.
\end{equation*}
Here $\bm n$ is the outward unit normal on $\partial\Omega$.
The finite element method seeks $\bm u_h\in V_h^1\subseteq H_0(\curl;\Omega)$ such
that
\begin{equation*}
  \inner[L^2(\Omega)]{\curl\bm u_h}{\curl\bm v_h}
  +\inner[L^2(\Omega)]{\bm u_h}{\bm v_h}
  =\inner[L^2(\Omega)]{\bm f}{\bm v_h}
  \qquad\forall\bm v_h\in V_h^1.
\end{equation*}
The reported quantity is the graph error
\[
  \norm[H(\curl;\Omega)]{\bm u-\bm u_h}
  =
  \left(
  \norm[L^2(\Omega)]{\bm u-\bm u_h}^2+
  \norm[L^2(\Omega)]{\curl(\bm u-\bm u_h)}^2
  \right)^{1/2}.
\]

On the unit square, the manufactured solution is
\begin{equation*}
  \bm u(x,y)
  =
  \begin{pmatrix}\sin(\pi y)\\ \sin(\pi x)\end{pmatrix},
  \qquad
  \bm f=(\pi^2+1)\bm u.
\end{equation*}
\begin{figure}[!htbp]
  \centering
  \includegraphics[width=0.66\linewidth]{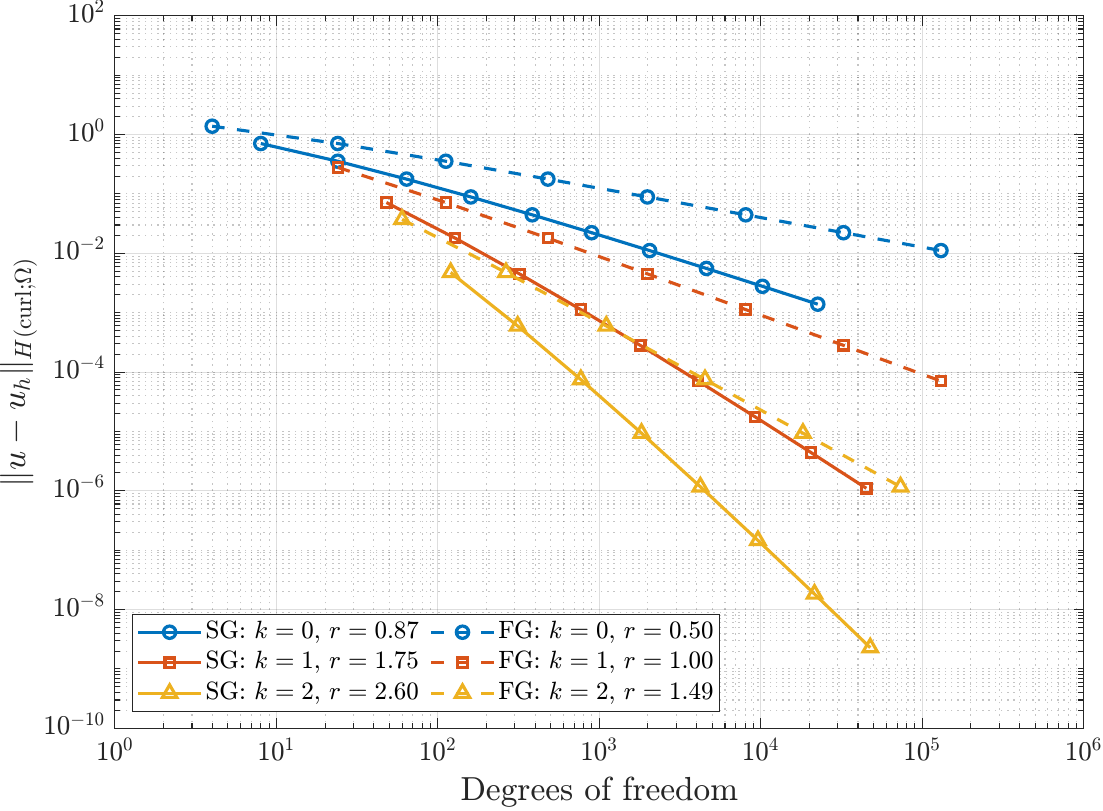}
  \caption{$H(\curl)$ error for the 2D source problem.
  Solid and dashed curves are sparse-grid and full-grid
  results, respectively.}
  \label{fig:source2d}
\end{figure}

On the unit cube, we use
\begin{equation*}
  \bm u(x,y,z)
  =
  \begin{pmatrix}
    \sin(\pi y)\sin(\pi z)\\
    \sin(\pi x)\sin(\pi z)\\
    \sin(\pi x)\sin(\pi y)
  \end{pmatrix},
  \qquad
  \bm f=(2\pi^2+1)\bm u.
\end{equation*}
Both fields are divergence free, satisfy the homogeneous tangential trace,
and obey the stated curl-curl identities.
\begin{figure}[!htbp]
  \centering
  \includegraphics[width=0.66\linewidth]{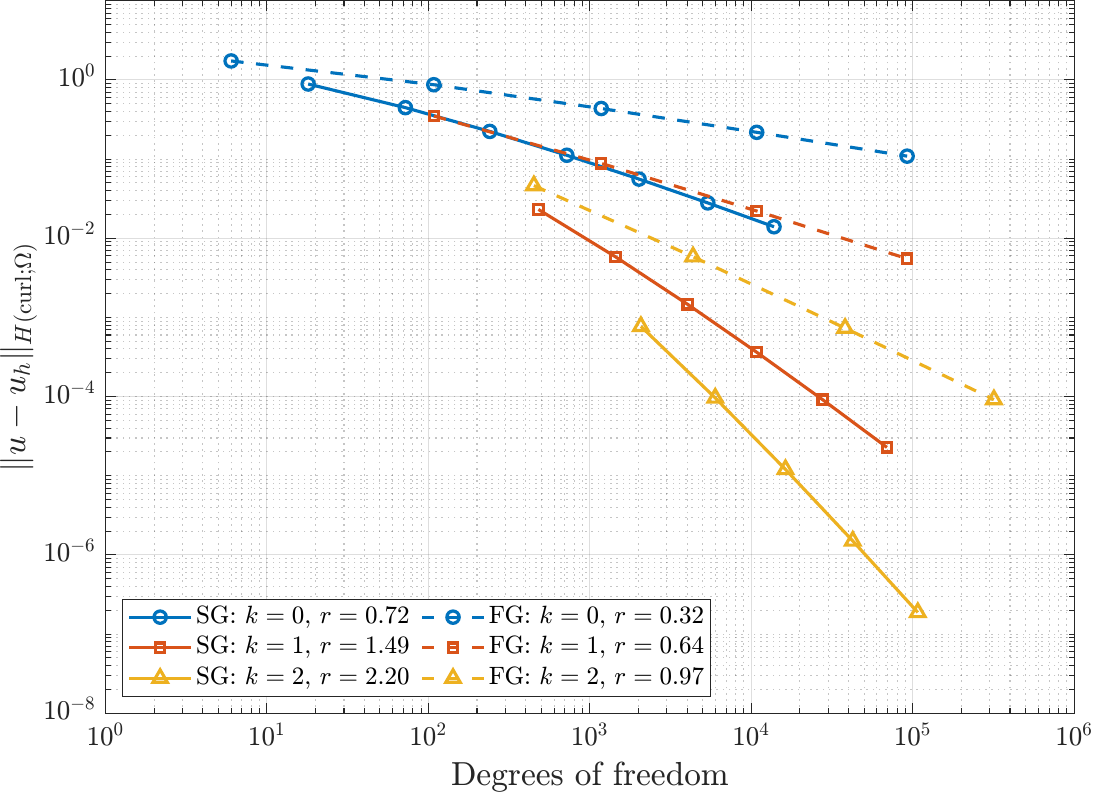}
  \caption{$H(\curl)$ error for the 3D source problem.
  Solid and dashed curves are sparse-grid and full-grid
  results, respectively.}
  \label{fig:source3d}
\end{figure}

Convergence histories are shown in Figures~\ref{fig:source2d} and~\ref{fig:source3d}.  For a smooth
solution,
Corollary~\ref{cor:graph-approximation} predicts a sparse-grid slope close to $k+1$
against degrees of freedom, up to logarithmic factors.  The corresponding
full-grid slope is $(k+1)/d$.  The measured full-grid rates agree closely
with these values.  The sparse rates increase by approximately one whenever
$k$ is increased by one; their modest deficit from $k+1$ is consistent with
the logarithmic factor and the finite level range.  


\subsection{Maxwell eigenproblem on a square with a hole} Although our analysis is performed only on a cube, it is instructive to test the numerical performance of the proposed sparse grid methods on domains with nontrivial topology. Let the computational domain be the square-annulus
\[
  \Omega_1
  =(0,1)^2\setminus[3/8,5/8]^2.
\]
The domain is partitioned into eight nonoverlapping rectangles.  Each box
carries its own sparse or full tensor hierarchy; tangential trace moments
are glued together on common interfaces, while the tangential degrees of
freedom on the outer and inner physical boundaries are set to zero.  We
compute the first ten positive eigenvalues of
\begin{equation*}
  \curl\curl\bm u=\lambda\bm u\quad\text{in }\Omega_1,
  \qquad
  \bm u\times\bm n=0\quad\text{on }\partial\Omega_1.
\end{equation*}

To remove zero eigenvalues associated with the large kernel of $\curl$, we introduce a Lagrange multiplier $p_h$ in the finite element space of 0-forms:
\[
  V_h^0
  =
  \begin{cases}
    V_{n,k}^0\cap H_0^1(\Omega_1),
      &\text{for a sparse-grid experiment},\\
    V_{n,k}^{0,\mathrm{FG}}\cap H_0^1(\Omega_1),
      &\text{for a full-grid experiment}.
  \end{cases}
\]
The finite element method seeks $\bm{0}\neq\bm{u}_h\in V_h^1$, $p_h\in V_h^0$ and $\lambda_h\geq0$ such that
\begin{align*}
  \inner[L^2(\Omega_1)]{\curl\bm u_h}{\curl\bm v_h}
  +\inner[L^2(\Omega_1)]{\bm v_h}{{\rm grad} p_h}
  &=\lambda_h\inner[L^2(\Omega_1)]{\bm u_h}{\bm v_h}
  \quad\forall\bm v_h\in V_h^1,\\
  \inner[L^2(\Omega_1)]{\bm u_h}{{\rm grad}q_h}&=0,\quad\forall q_h\in V_h^0.
\end{align*}
This treatment of the gradient kernel is standard in conforming edge-element
approximations of Maxwell eigenproblems
\citep{BoffiFernandesGastaldiPerugia1999,Boffi2010}. The eigenvalue $\lambda_h=0$ has multiplicity one and  corresponds to
a discrete harmonic vector field on the square annulus \cite{Li2024FoCM}. The ten smallest remaining positive eigenvalues are reported.

No closed form is available for the complete spectrum on this domain.  A
full-grid computation with $k=3$ and $n=7$ is used to produce the reference eigenvalues.  Its
first ten positive eigenvalues are listed in
Table~\ref{tab:reference-eigenvalues}.  If $\lambda_{h,j}$ denotes the $j$th
computed positive eigenvalue, the plotted error is
\begin{equation*}
  \max_{1\le j\le10}|\lambda_{h,j}-\lambda_j^{\mathrm{ref}}|/\lambda_j^{\mathrm{ref}}.
\end{equation*}

\begin{table}[!htbp]
  \centering
  \caption{Reference positive eigenvalues for the square-annulus problem
  ($k=3$, full-grid level $n=7$).}
  \label{tab:reference-eigenvalues}
  \begin{tabular}{r@{\qquad}r r@{\qquad}r}
    \toprule
    $j$ & $\lambda_j^{\mathrm{ref}}$ &
    $j$ & $\lambda_j^{\mathrm{ref}}$\\
    \midrule
    1 & 7.7647 & 6  & 46.194\\
    2 & 7.7647 & 7  & 46.793\\
    3 & 19.403 & 8 & 83.887\\
    4 & 36.765 & 9 & 83.887\\
    5 & 46.194 & 10 & 95.429\\
    \bottomrule
  \end{tabular}
\end{table}

\begin{figure}[!htbp]
  \centering
  \includegraphics[width=0.62\linewidth]{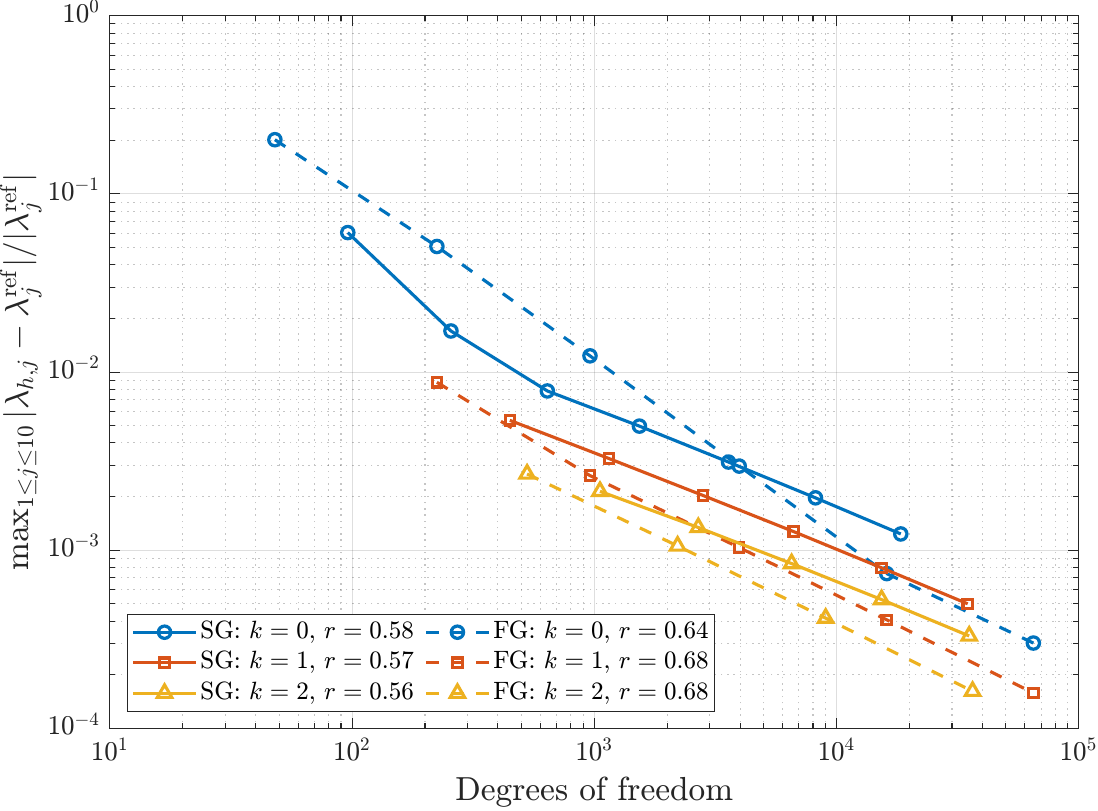}
  \caption{Maximum relative eigenvalue error among the first ten positive eigenvalues
  on the square annulus.  Solid curves denote sparse grids and dashed curves
  full grids.}
  \label{fig:eigen2d}
\end{figure}

Figure~\ref{fig:eigen2d} shows substantially weaker order dependence than the
smooth source experiments.  The fitted sparse-grid slopes are approximately
$0.55$--$0.58$, and the full-grid slopes approximately $0.64$--$0.68$ over
the displayed ranges.  Increasing $k$ mainly lowers the error constant and
does not materially change the slope.  This behavior is consistent with the
limited regularity of Maxwell eigenfunctions near the four reentrant inner
corners \citep[see, e.g.,][]{Monk2003,Boffi2010}.  Sparse approximation
additionally relies on mixed regularity, so
corner singularities can reduce its advantage relative to a full tensor
grid.  The experiment therefore illustrates that polynomial degree alone
cannot recover high-order convergence when the eigenspace lacks the
corresponding Sobolev and mixed Sobolev regularity.

\section{Concluding remarks}
\label{sec:conclusion}


We have constructed higher-order finite element differential forms on dyadic
tensor-product sparse grids using compatible one-dimensional Alpert multiwavelet
and integrated-Alpert hierarchies linked by differentiation. Tensorization
over a common downward-closed index set yields a conforming discrete de Rham
subcomplex and a commuting sparse-grid interpolant in arbitrary dimension and
form degree, with sparse-grid approximation error estimates under mixed Sobolev regularity. On the unit cube, an explicit
tensor-product homotopy preserves the sparse-grid spaces, proving discrete
exactness and yielding potentials uniformly stable in both grid level and
polynomial degree, and hence an $H(\dd)$-bounded commuting projection.

The numerical source experiments in two and three dimensions exhibit the predicted
order of convergence. For the square-annulus Maxwell eigenproblem, however,
reentrant-corner singularities limit the regularity of the eigenspaces and
substantially reduce the benefit of increasing the polynomial degree. Extensions
to mapped geometries and adaptive sparse-grid methods 
under low regularity remain natural directions for future work.

\section*{Declaration of AI use}
During the preparation of this work, the authors used ChatGPT to polish the writing,
generate code, and conduct computational experiments. The authors reviewed and edited all generated output and take full responsibility for the content of the paper.

\begingroup
\small
\bibliographystyle{plainnat}

\begin{thebibliography}{33}
\providecommand{\natexlab}[1]{#1}
\providecommand{\url}[1]{\texttt{#1}}
\expandafter\ifx\csname urlstyle\endcsname\relax
  \providecommand{\doi}[1]{doi: #1}\else
  \providecommand{\doi}{doi: \begingroup \urlstyle{rm}\Url}\fi

\bibitem[Alpert(1993)]{Alpert1993}
Bradley~K. Alpert.
\newblock A class of bases in {$L^2$} for the sparse representation of integral
  operators.
\newblock \emph{SIAM Journal on Mathematical Analysis}, 24\penalty0
  (1):\penalty0 246--262, 1993.
\newblock \doi{10.1137/0524016}.

\bibitem[Arnold(2018)]{Arnold2018}
Douglas~N. Arnold.
\newblock \emph{Finite Element Exterior Calculus}, volume~93 of \emph{CBMS-NSF
  Regional Conference Series in Applied Mathematics}.
\newblock Society for Industrial and Applied Mathematics, Philadelphia, 2018.
\newblock \doi{10.1137/1.9781611975543}.

\bibitem[Arnold and Awanou(2011)]{ArnoldAwanou2011}
Douglas~N. Arnold and Gerard Awanou.
\newblock The serendipity family of finite elements.
\newblock \emph{Foundations of Computational Mathematics}, 11\penalty0
  (3):\penalty0 337--344, 2011.
\newblock \doi{10.1007/s10208-011-9087-3}.

\bibitem[Arnold and Awanou(2014)]{ArnoldAwanou2014}
Douglas~N. Arnold and Gerard Awanou.
\newblock Finite element differential forms on cubical meshes.
\newblock \emph{Mathematics of Computation}, 83\penalty0 (288):\penalty0
  1551--1570, 2014.
\newblock \doi{10.1090/S0025-5718-2013-02783-4}.

\bibitem[Arnold et~al.(2006)Arnold, Falk, and Winther]{ArnoldFalkWinther2006}
Douglas~N. Arnold, Richard~S. Falk, and Ragnar Winther.
\newblock Finite element exterior calculus, homological techniques, and
  applications.
\newblock \emph{Acta Numerica}, 15:\penalty0 1--155, 2006.
\newblock \doi{10.1017/S0962492906210018}.

\bibitem[Arnold et~al.(2010)Arnold, Falk, and Winther]{ArnoldFalkWinther2010}
Douglas~N. Arnold, Richard~S. Falk, and Ragnar Winther.
\newblock Finite element exterior calculus: From {H}odge theory to numerical
  stability.
\newblock \emph{Bulletin of the American Mathematical Society}, 47\penalty0
  (2):\penalty0 281--354, 2010.
\newblock \doi{10.1090/S0273-0979-10-01278-4}.

\bibitem[Arnold et~al.(2015)Arnold, Boffi, and
  Bonizzoni]{ArnoldBoffiBonizzoni2015}
Douglas~N. Arnold, Daniele Boffi, and Francesca Bonizzoni.
\newblock Finite element differential forms on curvilinear cubic meshes and
  their approximation properties.
\newblock \emph{Numerische Mathematik}, 129\penalty0 (1):\penalty0 1--20, 2015.
\newblock \doi{10.1007/s00211-014-0631-3}.

\bibitem[Berglund(2014)]{Berglund2014}
Alexander Berglund.
\newblock Homological perturbation theory for algebras over operads.
\newblock \emph{Algebraic \& Geometric Topology}, 14\penalty0 (5):\penalty0
  2511--2548, 2014.
\newblock \doi{10.2140/agt.2014.14.2511}.

\bibitem[Boffi(2010)]{Boffi2010}
Daniele Boffi.
\newblock Finite element approximation of eigenvalue problems.
\newblock \emph{Acta Numerica}, 19:\penalty0 1--120, 2010.
\newblock \doi{10.1017/S0962492910000012}.

\bibitem[Boffi et~al.(1999)Boffi, Fernandes, Gastaldi, and
  Perugia]{BoffiFernandesGastaldiPerugia1999}
Daniele Boffi, Pedro Fernandes, Lucia Gastaldi, and Ilaria Perugia.
\newblock Computational models of electromagnetic resonators: Analysis of edge
  element approximation.
\newblock \emph{SIAM Journal on Numerical Analysis}, 36\penalty0 (4):\penalty0
  1264--1290, 1999.
\newblock \doi{10.1137/S003614299731853X}.

\bibitem[Boffi et~al.(2011)Boffi, Costabel, Dauge, Demkowicz, and
  Hiptmair]{Boffi2011}
Daniele Boffi, Martin Costabel, Monique Dauge, Leszek Demkowicz, and Ralf
  Hiptmair.
\newblock Discrete compactness for the p-version of discrete differential
  forms.
\newblock \emph{SIAM J. Numer. Anal.}, 49\penalty0 (1):\penalty0 135--158,
  2011.
\newblock \doi{10.1137/090772629}.

\bibitem[Bungartz and Griebel(2004)]{BungartzGriebel2004}
Hans-Joachim Bungartz and Michael Griebel.
\newblock Sparse grids.
\newblock \emph{Acta Numerica}, 13:\penalty0 147--269, 2004.
\newblock \doi{10.1017/S0962492904000182}.

\bibitem[Christiansen et~al.(2011)Christiansen, Munthe-Kaas, and
  Owren]{ChristiansenMuntheKaasOwren2011}
Snorre~H. Christiansen, Hans~Z. Munthe-Kaas, and Brynjulf Owren.
\newblock Topics in structure-preserving discretization.
\newblock \emph{Acta Numerica}, 20:\penalty0 1--119, 2011.
\newblock \doi{10.1017/S096249291100002X}.

\bibitem[D'Azevedo et~al.(2020)D'Azevedo, Green, and Mu]{DAzevedoGreenMu2020}
Eduardo~F. D'Azevedo, David~L. Green, and Lin Mu.
\newblock Discontinuous {G}alerkin sparse grids methods for time domain
  {Maxwell}'s equations.
\newblock \emph{Computer Physics Communications}, 256:\penalty0 107412, 2020.
\newblock \doi{10.1016/j.cpc.2020.107412}.

\bibitem[D{\~u}ng et~al.(2018)D{\~u}ng, Temlyakov, and
  Ullrich]{DungTemlyakovUllrich2018}
Dinh D{\~u}ng, Vladimir~N. Temlyakov, and Tino Ullrich.
\newblock \emph{Hyperbolic Cross Approximation}.
\newblock Advanced Courses in Mathematics---CRM Barcelona. Birkh{\"a}user,
  Cham, 2018.
\newblock \doi{10.1007/978-3-319-92240-9}.

\bibitem[Gillette et~al.(2019)Gillette, Kloefkorn, and
  Sanders]{GilletteKloefkornSanders2019}
Andrew Gillette, Tyler Kloefkorn, and Victoria Sanders.
\newblock Computational serendipity and tensor product finite element
  differential forms.
\newblock \emph{SMAI Journal of Computational Mathematics}, 5:\penalty0 1--21,
  2019.
\newblock \doi{10.5802/smai-jcm.41}.

\bibitem[Gr\u{a}dinaru and Hiptmair(2003{\natexlab{a}})]{GradinaruHiptmair2003}
Vasile Gr\u{a}dinaru and Ralf Hiptmair.
\newblock Mixed finite elements on sparse grids.
\newblock \emph{Numerische Mathematik}, 93\penalty0 (3):\penalty0 471--495,
  2003{\natexlab{a}}.
\newblock \doi{10.1007/s002110100382}.

\bibitem[Gr\u{a}dinaru and
  Hiptmair(2003{\natexlab{b}})]{GradinaruHiptmair2003Multigrid}
Vasile Gr\u{a}dinaru and Ralf Hiptmair.
\newblock Multigrid for discrete differential forms on sparse grids.
\newblock \emph{Computing}, 71\penalty0 (1):\penalty0 17--42,
  2003{\natexlab{b}}.
\newblock \doi{10.1007/s00607-003-0008-4}.

\bibitem[Gr\u{a}dinaru(2002)]{Gradinaru2002}
Vasile~Catrinel Gr\u{a}dinaru.
\newblock \emph{Whitney Elements on Sparse Grids}.
\newblock PhD thesis, Universit{\"a}t T{\"u}bingen, 2002.

\bibitem[Gugenheim et~al.(1991)Gugenheim, Lambe, and
  Stasheff]{GugenheimLambeStasheff1991}
V.~K. A.~M. Gugenheim, Larry~A. Lambe, and James~D. Stasheff.
\newblock Perturbation theory in differential homological algebra {II}.
\newblock \emph{Illinois Journal of Mathematics}, 35\penalty0 (3):\penalty0
  357--373, 1991.
\newblock \doi{10.1215/ijm/1255987784}.

\bibitem[Guo and Cheng(2016)]{GuoCheng2016}
Wei Guo and Yingda Cheng.
\newblock A sparse grid discontinuous {G}alerkin method for high-dimensional
  transport equations and its application to kinetic simulations.
\newblock \emph{SIAM Journal on Scientific Computing}, 38\penalty0
  (6):\penalty0 A3381--A3409, 2016.
\newblock \doi{10.1137/16M1060017}.

\bibitem[Guo and Cheng(2017)]{GuoCheng2017}
Wei Guo and Yingda Cheng.
\newblock An adaptive multiresolution discontinuous {G}alerkin method for
  time-dependent transport equations in multidimensions.
\newblock \emph{SIAM Journal on Scientific Computing}, 39\penalty0
  (6):\penalty0 A2962--A2992, 2017.
\newblock \doi{10.1137/16M1083190}.

\bibitem[Hiptmair(1999)]{Hiptmair1999}
Ralf Hiptmair.
\newblock Canonical construction of finite elements.
\newblock \emph{Mathematics of Computation}, 68\penalty0 (228):\penalty0
  1325--1346, 1999.
\newblock \doi{10.1090/S0025-5718-99-01166-7}.

\bibitem[Li(2024)]{Li2024FoCM}
Yuwen Li.
\newblock Nodal auxiliary space preconditioning for the surface de {R}ham
  complex.
\newblock \emph{Found. Comput. Math.}, 24:\penalty0 1019--1048, 2024.
\newblock \doi{10.1007/s10208-023-09611-0}.

\bibitem[Li and Zhang(2025)]{LiZhang2025}
Yuwen Li and Guozhi Zhang.
\newblock Higher order approximation rates for {ReLU} {CNN}s in {K}orobov
  spaces.
\newblock \emph{arXiv preprint}, page arXiv:2501.11275, 2025.

\bibitem[Li and Zhang(2026)]{LiZhang2026}
Yuwen Li and Guozhi Zhang.
\newblock Some super-approximation rates of {ReLU} neural networks for
  {K}orobov functions.
\newblock \emph{Commun. Math. Sci.}, 24\penalty0 (8):\penalty0 2159--2186,
  2026.

\bibitem[Mao and Zhou(2022)]{MaoZhou2022}
Tong Mao and Ding-Xuan Zhou.
\newblock Approximation of functions from {K}orobov spaces by deep
  convolutional neural networks.
\newblock \emph{Advances in Computational Mathematics}, 48\penalty0
  (6):\penalty0 84, 2022.
\newblock \doi{10.1007/s10444-022-09991-x}.

\bibitem[Monk(2003)]{Monk2003}
Peter Monk.
\newblock \emph{Finite Element Methods for Maxwell's Equations}.
\newblock Oxford University Press, Oxford, 2003.
\newblock \doi{10.1093/acprof:oso/9780198508885.001.0001}.

\bibitem[Montanelli and Du(2019)]{MontanelliDu2019}
Hadrien Montanelli and Qiang Du.
\newblock New error bounds for deep {ReLU} networks using sparse grids.
\newblock \emph{SIAM Journal on Mathematics of Data Science}, 1\penalty0
  (1):\penalty0 78--92, 2019.
\newblock \doi{10.1137/18M1189336}.

\bibitem[N{\'e}d{\'e}lec(1980)]{Nedelec1980}
Jean-Claude N{\'e}d{\'e}lec.
\newblock Mixed finite elements in {$\mathbb R^3$}.
\newblock \emph{Numerische Mathematik}, 35:\penalty0 315--341, 1980.
\newblock \doi{10.1007/BF01396415}.

\bibitem[Shen and Yu(2010)]{ShenYu2010}
Jie Shen and Haijun Yu.
\newblock Efficient spectral sparse grid methods and applications to
  high-dimensional elliptic problems.
\newblock \emph{SIAM Journal on Scientific Computing}, 32\penalty0
  (6):\penalty0 3228--3250, 2010.
\newblock \doi{10.1137/100787842}.

\bibitem[Tao et~al.(2019)Tao, Guo, and Cheng]{TaoGuoCheng2019}
Zhanjing Tao, Wei Guo, and Yingda Cheng.
\newblock Sparse grid discontinuous {G}alerkin methods for the
  {Vlasov--Maxwell} system.
\newblock \emph{Journal of Computational Physics: X}, 3:\penalty0 100022, 2019.
\newblock \doi{10.1016/j.jcpx.2019.100022}.

\bibitem[Wang et~al.(2016)Wang, Tang, Guo, and Cheng]{WangTangGuoCheng2016}
Zixuan Wang, Qi~Tang, Wei Guo, and Yingda Cheng.
\newblock Sparse grid discontinuous {G}alerkin methods for high-dimensional
  elliptic equations.
\newblock \emph{Journal of Computational Physics}, 314:\penalty0 244--263,
  2016.
\newblock \doi{10.1016/j.jcp.2016.03.005}.

\end{thebibliography}

\endgroup

\end{document}